\documentclass{article}

\usepackage[utf8]{inputenc}
\usepackage{amsmath, amssymb, amsthm, bm, mathrsfs}
\usepackage[cal=cm]{mathalpha}
\DeclareMathAlphabet{\dutchcal}{U}{dutchcal}{m}{n} % Define Dutchcal command
\usepackage{centernot}
\usepackage{braket}
\usepackage[top=1.5cm, bottom=1.5cm, left=2cm, right=2cm]{geometry}
\usepackage{csquotes, authblk, cases, enumitem}
\usepackage{graphicx, float, caption, subcaption, adjustbox}
\usepackage{tikz-cd}
\usepackage[hidelinks]{hyperref}
\usepackage{cleveref}
\usepackage[numbers,sort]{natbib}
\usepackage[section]{algorithm}
\usepackage{algpseudocode}
\usepackage{array, tabularx, xcolor, colortbl}
\usepackage{color, xcolor}
\usepackage[para]{footmisc}
\usepackage[normalem]{ulem}
\counterwithin{figure}{section}
\counterwithin{table}{section}
\numberwithin{equation}{section}
\graphicspath{{Figures/}}
\definecolor{darkred}{rgb}{0.8,0,0}
\DeclareMathOperator{\trace}{\operatorname{trace}}

\DeclareMathOperator{\re}{\operatorname{Re}}
\DeclareMathOperator{\im}{\operatorname{Im}}

\DeclareMathOperator{\nnz}{\operatorname{nnz}}

\DeclareMathOperator{\sparsity}{\operatorname{sp}}

\newcommand{\lcal}{\dutchcal{l}}
\newcommand{\rcal}{\dutchcal{r}}
\newenvironment{frcseries}{\fontfamily{frc}\selectfont}{}

\theoremstyle{definition}
\newtheorem{definition}{Definition}[section]
\newtheorem{remark}[definition]{Remark}
\newtheorem{example}[definition]{Example}
\newtheorem{theorem}[definition]{Theorem}
\newtheorem{lemma}[definition]{Lemma}
\newtheorem{corollary}[definition]{Corollary}

\title{Identifying Kronecker product factorizations}
\author[1]{Yannis Voet \thanks{yannis.voet@epfl.ch}}
\author[1]{Leonardo De Novellis \thanks{leonardo.denovellis@epfl.ch}}
\affil[1]{\small MNS, Institute of Mathematics, École polytechnique fédérale de Lausanne, Station 8, CH-1015 Lausanne, Switzerland}
\date{\today}

\begin{document}

\maketitle

\begin{abstract}
The Kronecker product is an invaluable tool for data-sparse representations of large networks and matrices with countless applications in machine learning, graph theory and numerical linear algebra. In some instances, the sparsity pattern of large matrices may already hide a Kronecker product. Similarly, a large network, represented by its adjacency matrix, may sometimes be factorized as a Kronecker product of smaller adjacency matrices. In this article, we determine all possible Kronecker factorizations of a binary matrix and visualize them through its decomposition graph. Such sparsity-informed factorizations may later enable good (approximate) Kronecker factorizations of real matrices or reveal the latent structure of a network. The latter also suggests a natural visualization of Kronecker graphs. \\

\end{abstract}

\section{Introduction}
This work focuses on finding all possible Kronecker product factorizations of a binary (or Boolean) matrix $A$. That is, given $A \in \{0,1\}^{n \times n}$, we seek all possible ways of writing
\begin{equation}
\label{eq: factorization}
A = \bigotimes_{i=1}^{\ell} A_i  
\end{equation}
for $\ell > 1$ factor matrices $A_i \in \{0,1\}^{n_i \times n_i}$, $n_i > 1$, whose sizes satisfy $\prod_{i=1}^\ell n_i = n$.

More than 25 years ago, Van Loan \cite{van2000ubiquitous} rightly predicted that the Kronecker product would emerge as an increasingly useful tool in scientific computing. The problem stated in \eqref{eq: factorization} is one of many instances of Kronecker products and is at the forefront of several important applications. For instance, if $A$ is the adjacency matrix of an unweighted graph $G$ representing a network, decomposing it as a Kronecker product effectively represents the graph as a Kronecker product of smaller graphs $G_i$. Factorizing a large graph into smaller graphs allows compressing the network and often helps both analyzing and visualizing it \cite{leskovec2007scalable,leskovec2010kronecker,kepner2011graph}. For those reasons, graph products, and in particular the Kronecker product (sometimes also called direct, cardinal, categorical or tensor product) have been well studied by the graph theory community \cite{imrich1998factoring,hammack2011handbook}. A graph $G$ is a Kronecker product if there exists a graph $G'$ isomorphic to $G$ such that its adjacency matrix is a Kronecker product \cite{weichsel1962kronecker}. In other words, there exists a permutation matrix $P$ such that $P^TAP$ is a Kronecker product \cite{calderoni2021direct}. This definition implies that identifying Kronecker product graphs is tightly intertwined with graph isomorphism problems, which are notoriously difficult. Additionally, computational tools for identifying Kronecker products graphs are still rather primitive. Imrich et al. \cite{imrich1998factoring} mostly studied theoretical questions of existence and uniqueness of graph factorizations (also for other products) and proposed a polynomial time algorithm for factorizing connected non-bipartite graphs. Recognizing the limitations of Imrich's algorithm, Calderoni et al. \cite{calderoni2023heuristic} recently proposed a heuristic for factorizing general graphs. 

Apart from the theoretical and computational challenges of graph factorizations, Kronecker product graphs play an important role in modeling large networks. In the machine learning community, Leskovec et al. \cite{leskovec2007scalable,leskovec2010kronecker} realized that Kronecker product graphs successfully mimic several key properties of real networks and proposed a Kronecker graph model by taking successive products with an initiator graph. In this context, the problem consists in fitting a real network to a Kronecker graph instead of exactly factorizing it.

The problem addressed in this contribution is related albeit different from the graph factorization problem. Here we assume that the binary matrix $A$ is given and directly attempt to factorize it. On the one hand, this eliminates the hurdle of finding a permutation of $A$ such that it becomes factorizable. On the other hand, we must still cope with potential problems of non-uniqueness. A prominent application of our research lies in (approximately) factorizing large sparse matrices. Given \emph{fixed} sizes $n_1$ and $n_2$ such that $n=n_1n_2$, Van Loan and Pitsianis \cite{van1993approximation} presented in the early 1990s an algorithm for computing the best Kronecker product factorization of a general matrix $A \in \mathbb{R}^{n \times n}$ in the Frobenius norm. In other words, they find the best factor matrices $B_i \in \mathbb{R}^{n_i \times n_i}$ such that
\begin{equation}
\label{eq: nkp_problem}
    \|B - B_1 \otimes B_2 \|_F
\end{equation}
is minimized. Their algorithm has since then been generalized to Kronecker products of an arbitrary number of matrices $B_i$ for $i=1,\dots,\ell$ \cite{langville2004akronecker}, although the resulting factorization may no longer be optimal. Unfortunately, neither Van Loan's original algorithm nor subsequent generalizations are easy to apply in a ``black-box'' fashion. Indeed, some of the key parameters are the sizes $n_i$ of the factor matrices $A_i$ and their number $\ell$. Even for a fixed factorization length $\ell$, there generally exist several tuples of sizes $\bm{n}=(n_1,n_2,\dots,n_\ell)$ satisfying $n_1n_2\cdots n_\ell=n$ and without any background information on the problem's origin, choosing the ``right'' length $\ell$ and the ``right'' sizes $\bm{n}$ is not obvious at all. By ``right'' we mean the factorization that minimizes some objective function or error measure such as the Frobenius norm in \eqref{eq: nkp_problem}. Clearly, Van Loan's algorithm (and subsequent generalizations) will deliver an approximate factorization for arbitrary choices of length $\ell$ and compatible sizes $\bm{n}$. However, it may produce a very poor ``approximation'' unless $\ell$ and $\bm{n}$ are suitably chosen. Since such approximations are oftentimes used as preconditioners within iterative solvers \cite{van1993approximation,langville2004akronecker,nagy2006kronecker,voet2025preconditioning}, it could have rippling effects down the computational pipeline. Instead, depending on the origin of the matrix, there often exist natural choices for $\ell$ and $\bm{n}$ that will deliver remarkably good factorizations. In some instances, those choices are naturally encoded in the sparsity pattern of the matrix. For example, matrices arising from tensorized finite element discretizations of partial differential equations (PDEs) often hide a Kronecker product in their sparsity pattern, although the matrix itself is rarely exactly factorizable \cite{hofreither2018black,voet2025mass}. Thus, our work will not attempt to factorize the matrices themselves but only their sparsity pattern encoded by a binary matrix. Clearly, if a real matrix $B$ already happens to be a Kronecker product, then so is its sparsity pattern $A$. Although the converse does not hold, the structure encoded in the sparsity pattern may still suggest suitable sizes for approximately factorizing $B$. We will therefore assume that the matrices one attempts to (approximately) factorize are sparse. We note that the specific problem of factorizing permutations has already been addressed in \cite{egner1997decomposing}, where an algorithm based on group theory arguments is also presented. Conversely, for completely dense matrices, the sparsity pattern does not help and one may need to compute all factorizations and retain the best one, as was suggested in \cite{cai2022kopa} for applications in image analysis.

In this article, we present a theory and algorithm for computing all possible factorizations of the form \eqref{eq: factorization}. In particular, the algorithm finds suitable lengths $\ell$ and sizes $\bm{n}$ that may be directly supplied as input parameters to Van Loan's algorithm. In general, one is interested in computing the factorization of greatest length since it yields the greatest compression.

The outline for the rest of the article is as follows: after setting up the problem in \Cref{se: problem_description} and recalling some basic properties of the Kronecker product, we propose in \Cref{se: decomposable_matrices} a fast algorithm for determining whether a sparse binary matrix $A$ is factorizable for \emph{fixed} sizes $(n_1,n_2)$. Remarkably, as shown later in the same section, factorizations of length $\ell=2$ completely describe the structure of $A$ and enable finding factorizations of length $\ell > 2$. Subsequently, we present in \Cref{se: decomposition_graph} the decomposition graph, a simple and elegant way of visualizing all possible Kronecker factorizations of $A$. \Cref{se: applications} is then devoted to applications. The first one pertains to space-time discretizations of PDEs, whose coefficient matrix is approximated by a length $4$ Kronecker product. The second one exploits the Kronecker structure of the adjacency matrix of a network for visualizing Kronecker graphs. Lastly, the third one identifies a Kronecker product structure within a subsequence of unitary operations encoding a quantum gate. Finally, conclusions are stated in \Cref{se: conclusion}.

\section{Problem description}
\label{se: problem_description}
In this section, we introduce some first definitions and preliminary results to set up the problem. Throughout this article, we assume that $A \in \mathbb{B}^{n \times n}$ is a large sparse binary matrix, where $\mathbb{B}=\{0,1\}$. This matrix may for instance define the adjacency matrix of a large network or encode the sparsity pattern of a matrix with real or complex entries. The sparsity pattern of a matrix $X \in \mathbb{R}^{n \times n}$, denoted $\sparsity(X)$, simply keeps track of the position of nonzero entries:
\begin{equation*}
    \sparsity(X)=\{(i,j) \colon x_{ij} \neq 0,\; 1 \leq i,j \leq n\}.
\end{equation*}
The binary matrix $A$ is then defined as
\begin{equation*}
a_{ij}=
    \begin{cases}
        1 & \text{if } (i,j) \in \sparsity(X), \\
        0 & \text{if } (i,j) \notin \sparsity(X).
    \end{cases}
\end{equation*}
When handling binary variables, the addition ``$+$'' and multiplication ``$\cdot$'' refer to the standard Boolean addition and multiplication, whose truth table is recalled in \Cref{fig: boolean_addition,fig: boolean_multiplication}, respectively.

\begin{table}[H]
     \centering
     \begin{subtable}[t]{0.48\textwidth}
    \centering
    \begin{tabular}{|m{1cm}|m{1cm}|m{1cm}|}
    \hline
     $x$ & $y$ & $x+y$ \\
     \hline
     0 & 0 & 0\\
     0 & 1 & 1\\
     1 & 0 & 1\\
     1 & 1 & 1\\
    \hline
    \end{tabular}
    \caption{Boolean addition}
    \label{fig: boolean_addition}
     \end{subtable}
     \hfill
     \begin{subtable}[t]{0.48\textwidth}
    \centering
    \begin{tabular}{|m{1cm}|m{1cm}|m{1cm}|}
    \hline
     $x$ & $y$ & $x \cdot y$ \\
     \hline
     0 & 0 & 0\\
     0 & 1 & 0\\
     1 & 0 & 0\\
     1 & 1 & 1\\
    \hline
    \end{tabular}
    \caption{Boolean multiplication}
    \label{fig: boolean_multiplication}
     \end{subtable}
     \hfill
    \caption{Boolean operations}
    \label{fig: boolean_operations}
\end{table}

Since $\sparsity(X+Y) \subseteq \sparsity(X) \cup \sparsity(Y)$, for two matrices $X,Y \in \mathbb{R}^{n \times n}$, the Boolean addition is a natural choice in this context and simply describes the worst case fill-in when summing two sparse matrices. Also note that the Boolean multiplication reduces to the standard multiplication of real or complex numbers. However, the operation that is really at the center of our attention is the \emph{Kronecker product}. Given two matrices $X \in \mathbb{R}^{n_1 \times n_1}$ and $Y \in \mathbb{R}^{n_2 \times n_2}$, their Kronecker product $X \otimes Y \in \mathbb{R}^{n_1n_2 \times n_1n_2}$ is defined as the block matrix
\begin{equation*}
X \otimes Y = 
\begin{pmatrix}
x_{11}Y & x_{12}Y & \cdots & x_{1n_1}Y \\
x_{21}Y & x_{22}Y & \cdots & x_{2n_1}Y \\
\vdots  & \vdots  & \ddots & \vdots \\
x_{n_11}Y & x_{n_12}Y & \cdots & x_{n_1n_1}Y
\end{pmatrix}.
\end{equation*}
The Kronecker product satisfies some basic properties of associativity and distributivity, whose proof is commonly found in standard textbooks (see e.g. \cite{horn1991topics}). Given matrices $X,Y,Z$ of conforming size,
\begin{align*}
    (X \otimes Y) \otimes Z &= X \otimes (Y \otimes Z) & &\text{(associativity)}, \\
    X \otimes (Y + Z) &= X \otimes Y + X \otimes Z & &\text{(distributivity I)}, \\
    (Y + Z) \otimes X &= Y \otimes X + Z \otimes X & &\text{(distributivity II)}. 
\end{align*}
Obviously, the same definition and properties hold for binary matrices, provided one substitutes the standard addition and multiplication with their Boolean counterpart. Among the binary matrices encountered in applications, decomposable ones are of great interest.

\begin{definition}[Decomposable matrix]
\label{def: decomposable_mat}
A matrix $A \in \mathbb{B}^{n \times n}$ is called \emph{decomposable} (with respect to the Kronecker product) if there exists an integer $\ell>1$ and factor matrices $A_i \in \mathbb{B}^{n_i \times n_i}$ with $n_i > 1$ such that
\begin{equation}
\label{eq: decomposition}
A = \bigotimes_{i=1}^{\ell} A_i.
\end{equation}
The integer $\ell$ is called the \emph{length} of the factorization.
\end{definition}

\begin{definition}[Prime matrix]
A matrix $A \in \mathbb{B}^{n \times n}$ is called \emph{prime} if it cannot be decomposed.    
\end{definition}

In particular, note that $A$ is necessarily prime if $n$ is prime. Among all possible decompositions of $A$, so-called \emph{prime decompositions} are particularly important in the forthcoming discussion.

\begin{definition}[Prime decomposition]
A decomposition \eqref{eq: decomposition} is called \emph{prime} if all its factor matrices are prime. 
\end{definition}

In other words, prime decompositions cannot be further expanded into a factorization of greater length. A decomposition (or factorization) of $A$ is commonly identified with the tuple $\bm{n}=(n_1,n_2,\dots,n_\ell)$ specifying the sizes of the factor matrices in the decomposition. The size of the matrix $A$ is then read from the product $n=\Pi(\bm{n}):=\Pi_{i=1}^\ell n_i$. The next lemma shows that this identification does not cause any ambiguity in all practically relevant cases.

\begin{lemma}
\label{lem: uniqueness}
For fixed sizes $(n_1,n_2,\dots,n_\ell)$, the factorization of $A \in \mathbb{B}^{n \times n} \setminus \{0\}$, if it exists, is unique.
\end{lemma}
\begin{proof}
Assume that $A$ admits two $(n_1,n_2,\dots,n_\ell)$ factorizations with factor matrices $\{A_i\}_{i=1}^\ell$ and $\{\hat{A}_i\}_{i=1}^\ell$. Then,
\begin{equation*}
    A = A_1 \otimes A_2 \otimes \dots \otimes A_{\ell-1} \otimes A_\ell = \hat{A}_1 \otimes \hat{A}_2 \otimes \dots \otimes \hat{A}_{\ell-1} \otimes \hat{A}_\ell.
\end{equation*}
From the associativity of the Kronecker product, the above relation may be rewritten
\begin{equation*}
    X \otimes A_\ell = \hat{X} \otimes \hat{A}_\ell
\end{equation*}
where $X=A_1 \otimes A_2 \otimes \dots \otimes A_{\ell-1}$ and $\hat{X}=\hat{A}_1 \otimes \hat{A}_2 \otimes \dots \otimes \hat{A}_{\ell-1}$. Thus, $A$ is an $n_1n_2 \dots n_{\ell-1} \times n_1n_2 \dots n_{\ell-1}$ block matrix with blocks of size $n_\ell \times n_\ell$. Its $(i,j)$th block $A_{i,j}$ is given by
\begin{equation*}
    A_{i,j} = x_{ij}A_\ell = \hat{x}_{ij}\hat{A}_\ell \quad i,j=1,\dots,n_1n_2 \dots n_{\ell-1}.
\end{equation*}
Recalling that $X,A_\ell$ and $\hat{X},\hat{A}_\ell$ are binary matrices and none of them can be zero (since $A \neq 0$), two cases must be distinguished:
\begin{itemize}[noitemsep]
    \item If $A_{i,j} \neq 0$, then $x_{ij}=\hat{x}_{ij}=1$ and $A_\ell=\hat{A}_\ell$.
    \item If $A_{i,j} = 0$, then $x_{ij}=\hat{x}_{ij}=0$.
\end{itemize}
In either case $x_{ij}=\hat{x}_{ij}$ and consequently, $X=\hat{X}$. Moreover, since $A \neq 0$, the first case is necessarily encountered at least once and yields $A_\ell=\hat{A}_\ell$. Recursively applying the same arguments on $X=\hat{X}$ proves that $A_i=\hat{A}_i$ for all $i=1,\dots,\ell$.
\end{proof}

Note that \Cref{lem: uniqueness} does not hold for real valued matrices due to the scaling indeterminacy. Indeed, for $X \in \mathbb{R}^{n_1 \times n_1}$ and $Y \in \mathbb{R}^{n_2 \times n_2}$, $X \otimes Y = (\alpha X) \otimes (\alpha^{-1}Y)$ for any $\alpha \neq 0$. For binary matrices, \Cref{lem: uniqueness} eliminates this problem and allows identifying a factorization with the size of its factors. Nevertheless, as shown in the next couple of examples, a given matrix $A \in \mathbb{B}^{n \times n}$ may still have multiple distinct factorizations with factor matrices of different size.

\begin{example}
\label{ex: max_mat_12}
From the following string of identities
\begin{align*}
A &=
\begin{pmatrix}
1 & 1 \\
0 & 0
\end{pmatrix}
\otimes
\begin{pmatrix}
1 & 1 \\
0 & 0
\end{pmatrix}
\otimes
\begin{pmatrix}
1 & 1 & 1 \\
0 & 0 & 0 \\
1 & 1 & 1
\end{pmatrix} \\
&=
\begin{pmatrix}
1 & 1 \\
0 & 0
\end{pmatrix}
\otimes
\begin{pmatrix}
1 & 1 & 1 \\
1 & 1 & 1 \\
0 & 0 & 0
\end{pmatrix}
\otimes
\begin{pmatrix}
1 & 1 \\
0 & 0
\end{pmatrix} \\
&=
\begin{pmatrix}
1 & 1 & 1 \\
0 & 0 & 0 \\
0 & 0 & 0
\end{pmatrix}
\otimes
\begin{pmatrix}
1 & 1 \\
1 & 1
\end{pmatrix}
\otimes
\begin{pmatrix}
1 & 1 \\
0 & 0
\end{pmatrix},
\end{align*}
it appears that $A$ admits $(2,2,3)$, $(2,3,2)$ and $(3,2,2)$ factorizations. Moreover, those factorizations are all prime since both $2$ and $3$ are prime.
\end{example}

Interestingly, decomposable matrices may have prime factorizations of different length.

\begin{example}
\label{ex: non_max_mat_12}
From the equalities
\begin{align*}
A &=
\begin{pmatrix}
1 & 0 & 0 \\
0 & 0 & 0 \\
0 & 0 & 0
\end{pmatrix}
\otimes
\begin{pmatrix}
1 & 0 & 0 & 0 \\
0 & 0 & 0 & 0 \\
0 & 0 & 0 & 0 \\
1 & 0 & 0 & 0
\end{pmatrix} \\
&=
\begin{pmatrix}
1 & 0 \\
0 & 0
\end{pmatrix}
\otimes
\begin{pmatrix}
1 & 0 \\
1 & 0
\end{pmatrix}
\otimes
\begin{pmatrix}
1 & 0 & 0 \\
0 & 0 & 0 \\
0 & 0 & 0
\end{pmatrix},
\end{align*}
we conclude that $A$ admits $(3,4)$ and $(2,2,3)$ factorizations. Note that the first factorization is prime since the $4 \times 4$ matrix cannot be further decomposed into a Kronecker product of $2 \times 2$ matrices and the second factorization is evidently prime since $2$ and $3$ are both prime.
\end{example}

Among all possible factorizations of $A$, we are particularly interested in those of greatest length. Such factorizations are relevant in many applications since they generally provide the greatest compression. Indeed, storing the factor matrices $\{A_i\}_{i=1}^\ell$ only requires a fraction of the memory for $A$. The same holds true when subsequently (approximately) factorizing a real matrix $B$ having the same sparsity pattern as $A$. In the next section, we explain how to find all length $2$ factorizations of a binary matrix $A$ and how to eventually combine them to produce factorizations of greater length.

\section{Decomposable matrices}
\label{se: decomposable_matrices}
As explained later in the article, factorizations of length $\ell=2$ allow finding factorizations of length $\ell>2$. Thanks to this remarkable property, we may exclusively focus on the former, which drastically simplifies the problem. In this case, if $A$ admits an $(n_1,n_2)$ factorization, then $n_1$ and $n_2$ are divisors of $n$. Denoting $[n] \subset \mathbb{N}$ the set of positive divisors of $n$, a pair of integers $(n_1,n_2)$ with $n_1, n_2 \in [n] \setminus \{1,n\}$ and $n_1n_2=n$ is called a \emph{compatible pair} and we denote $\dutchcal{C}$ the set of all compatible pairs. Clearly, there are as many compatible pairs as there are non-trivial divisors of $n$ and this number is related to the prime decomposition of $n$. Recall that for any natural integer $n \in \mathbb{N}$, there exist prime numbers $p_1,\dots,p_m$ and integer exponents $k_1,\dots,k_m$ such that
\begin{equation*}
    n = \prod_{j=1}^m p_j^{k_j}.
\end{equation*}
The number of divisors of $n$ is therefore $\prod_{j=1}^m (k_j+1)$. Excluding the trivial divisors $1$ and $n$, the cardinality of $\dutchcal{C}$ is
\begin{equation*}
    |\dutchcal{C}| = \prod_{j=1}^m (k_j+1) - 2.
\end{equation*}
The set $\dutchcal{C}$ contains all pairs of candidate sizes for attempting a factorization. Decomposable matrices are factorizable for at least one compatible pair and clearly form a very special subset of binary matrices. As a matter of fact, the next lemma shows that most binary matrices one could form are in fact prime (even if $n$ is not prime).

\begin{lemma}
\label{lem: random_matrix_factorizability}
Let $A \in \mathbb{B}^{n \times n}$. Then,
\begin{equation*}
    \mathbb{P}[A \text{ is factorizable}] \leq \sum_{(n_1,n_2) \in \dutchcal{C}} \frac{2^{n_1^1+n_2^2}}{2^{n^2}}.
\end{equation*}
\end{lemma}
\begin{proof}
Let $A \in \mathbb{B}^{n \times n}$ have i.i.d. $\operatorname{Be}(1/2)$ random entries (i.e. independent and identically distributed random variables that take value $1$ or $0$, each with probability $1/2$). The number of binary matrices $A$ is $2^{n^2}$ while the number of $(n_1,n_2)$ factorizable matrices for a given compatible pair $(n_1,n_2)$ is $2^{n_1^1+n_2^2}$. Thus, the probability that $A$ admits an $(n_1,n_2)$ factorization is $2^{n_1^1+n_2^2}/2^{n^2}$. The probability that $A$ is factorizable for any compatible pair is then obtained by bounding the probability of the union by the sum of probabilities.
\end{proof}

This article would not make much sense if we were to consider random binary matrices. However, for specific applications, the sparsity pattern of binary matrices often hides a Kronecker product. Our algorithm for detecting it relies on a very special subset of matrices that are factorizable for any compatible pair.

\begin{definition}[Maximal matrix]
A decomposable matrix $A \in \mathbb{B}^{n \times n}$ is called \emph{maximal} if it admits an $(n_1,n_2)$ factorization for any compatible pair $(n_1,n_2) \in \dutchcal{C}$.
\end{definition}

So far, the only means of checking whether a matrix is maximal amounts to verifying that it is factorizable for any compatible pair $(n_1,n_2) \in \dutchcal{C}$. Some simple examples of maximal matrices include the identity matrix and the matrix of all ones. One can also easily verify that the matrix $A$ in \Cref{ex: max_mat_12} is maximal. Other less obvious examples are the so-called \emph{basis matrices}.

\begin{definition}[Basis matrix]
The \emph{basis matrices} $\{E_{ij}\}_{i,j=1}^n \subset \mathbb{B}^{n \times n}$ are defined as
\begin{equation*}
    (E_{ij})_{kl} = \delta_{ik}\delta_{jl}
\end{equation*}
where the Kronecker delta symbol is defined as $\delta_{ij}=1$ if $i=j$ and $\delta_{ij}=0$ if $i \neq j$.
\end{definition}

In other words, the basis matrix $E_{ij}$ has a single component equal to $1$ in position $(i,j)$ and $0$ elsewhere. Over the real (or complex) field, the set of all basis matrices $\{E_{ij}\}_{i,j=1}^n$ forms the so-called canonical basis for the vector space of real (or complex) matrices of size $n$. However, they are brought up here for another reason: decomposable basis matrices are maximal. This result is the object of the next lemma whose constructive proof is central to our forthcoming analysis.

\begin{lemma}
\label{lem: maximal_basis_matrix}
All decomposable basis matrices $E_{ij}$ are maximal for $i,j=1,\dots,n$.
\end{lemma}
\begin{proof}
Firstly, since $E_{ij}$ is decomposable, $n$ cannot be prime and it has at least one non-trivial divisor. Let $n_2 \in [n] \setminus \{1,n\}$. Then, the Euclidean division of $i-1$ and $j-1$ by $n_2$ yields
\begin{align*}
    i-1&=\tilde{i}_1 n_2 + \tilde{i}_2, \\
    j-1&=\tilde{j}_1 n_2 + \tilde{j}_2,
\end{align*}
for integers $0 \leq \tilde{i}_1,\tilde{j}_1 < n_1 = d/n_2$ and $0 \leq \tilde{i}_2,\tilde{j}_2 < n_2$. Finally, setting $i_k=\tilde{i}_k+1$ and $j_k = \tilde{j}_k+1$ for $k=1,2$, we obtain
\begin{equation*}
    E_{ij} = E_{i_1 j_1} \otimes E_{i_2 j_2}
\end{equation*}
for basis matrices $E_{i_1 j_1} \in \mathbb{B}^{n_1 \times n_1}$ and $E_{i_2 j_2} \in \mathbb{B}^{n_2 \times n_2}$. Thus, $E_{ij}$ admits an $(n_1,n_2)$ factorization. Since this construction holds for any divisor $n_2 \in [n] \setminus \{1,n\}$, $E_{ij}$ is maximal.
\end{proof}

\Cref{lem: maximal_basis_matrix} lays the foundation of our method and allows us to easily check whether \emph{any} binary matrix $A \in \mathbb{B}^{n \times n}$ is factorizable. Some interesting ideas in that direction were drafted in \cite{wei2023kronecker} but the authors reached erroneous conclusions. The next few paragraphs present a complete and correct algorithm. Firstly, the matrix is expanded in terms of basis matrices as
\begin{equation}
\label{eq: basis_representation}
    A = \sum_{(i,j) \in \sparsity(A)} E_{ij}.
\end{equation}
Secondly, given a compatible pair $(n_1,n_2)$, each $E_{ij}$ is factorized as $E_{ij} = E_{i_1 j_1} \otimes E_{i_2 j_2}$ with $1 \leq i_1,j_1 \leq n_1$ and $1 \leq i_2,j_2 \leq n_2$ as done in the proof of \Cref{lem: maximal_basis_matrix}. Note that the only information that matters in \eqref{eq: basis_representation} are the integer pairs $(i_1,j_1)$ and $(i_2,j_2)$, not the actual matrices. Moreover, for convenience, we map each index pair $(i_1,j_1)$ and $(i_2,j_2)$ to a linear index. For a matrix of size $n$, $l_n(i,j)=(j-1)n+i$ is the linear index corresponding to $(i,j)$. For instance, the linear indices corresponding to the basis matrices $\{E_{11}, E_{21}, E_{12}, E_{22}\} \subset \mathbb{B}^{2 \times 2}$ are $\{1,2,3,4\}$ (in the same order). Reverting back to index pairs and basis matrices is also straightforward from the Euclidean division. To lighten the notation, we drop the dependency on the sizes and simply denote
\begin{align*}
    l_1&=l_{n_1}(i_1,j_1) \\
    l_2&=l_{n_2}(i_2,j_2)
\end{align*}
the linear indices for the index pairs $(i_1,j_1)$ and $(i_2,j_2)$, respectively. A superscript $k$ is then appended to linear indices and index pairs for identifying the $k$th term in the sum \eqref{eq: basis_representation}. Finally, we define the set of pairs of linear indices
\begin{equation*}
    S = \{(l_1^k,l_2^k) \colon k=1,\dots,\nnz(A)\},
\end{equation*}
where $\nnz(A)$ denotes the number of non-zero entries of $A$. Our algorithm then uses the following equivalence for determining whether $A$ is factorizable.

\begin{lemma}
\label{lem: equivalency}
A matrix $A \in \mathbb{B}^{n \times n}$ admits an $(n_1,n_2)$ factorization if and only if there exist subsets $S_1,S_2 \subset \mathbb{N}$ such that
\begin{equation*}
    S = S_1 \times S_2.
\end{equation*}
\end{lemma}
\begin{proof}
The matrix $A$ is $(n_1,n_2)$ factorizable if and only if
\begin{equation}
\label{eq: fact_A}
    A = \left(\sum_{(i_1,j_1) \in I_1} E_{i_1j_1}\right) \otimes \left(\sum_{(i_2,j_2) \in I_2} E_{i_2j_2}\right)
\end{equation}
for subsets
\begin{equation*}
    I_1 \subseteq \{(i_1,j_1) \colon 1 \leq i_1,j_1 \leq n_1\}, \qquad I_2 \subseteq \{(i_2,j_2) \colon 1 \leq i_2,j_2 \leq n_2\}.
\end{equation*}
After mapping the index pairs within $I_1$ and $I_2$ to linear index sets $S_1$ and $S_2$, respectively, the matrix relation \eqref{eq: fact_A} is then equivalent to the set relation $S = \{(l_1,l_2) \colon l_1 \in S_1,\, l_2 \in S_2\} = S_1 \times S_2$ on linear indices.
\end{proof}

Thanks to \Cref{lem: equivalency}, our algorithm checks whether $A$ is factorizable by attempting to write $S$ as a Cartesian product. The latter merely requires ordering the elements of $S$ and is efficiently coded up. Thus, we have substituted the matrix problem with a scalar one that only uses the position of nonzero entries rather than the matrix itself. Once $S_1$ and $S_2$ have been found, reverting back to index pairs yields $I_1$ and $I_2$, the sparsity patterns of the factor matrices constituting the product. The next couple of examples illustrate our argument.

\begin{example}
We would like to determine whether the binary matrix
\begin{equation*}
    A =
    \begin{pmatrix}
        1 & 0 & 0 & 0 \\
        0 & 1 & 0 & 0 \\
        0 & 0 & 1 & 0 \\
        0 & 0 & 0 & 0
    \end{pmatrix}
\end{equation*}
is factorizable for the (only) compatible pair $(2,2)$. By inspection, the answer is trivially negative but we will show it algebraically from the procedure described above. Firstly, the sparsity pattern of $A$ is
\begin{equation*}
    \sparsity(A) = \{(1,1), (2,2), (3,3)\}.
\end{equation*}
Secondly, for the $(2,2)$ compatible pair, we obtain the sets
\begin{align*}
    \{(i_1^k,j_1^k)\}_{k=1}^3 &= \{(1,1),(1,1),(2,2)\}, \\
    \{(i_2^k,j_2^k)\}_{k=1}^3 &= \{(1,1),(2,2),(1,1)\},
\end{align*}
which are mapped to linear indices
\begin{equation*}
    S = \{(l_1^k,l_2^k)\}_{k=1}^3 = \{(1,1),(1,4),(4,1)\}.
\end{equation*}
This set cannot be expressed as a Cartesian product and consequently $A$ is not $(2,2)$ factorizable. Consider now the slightly modified matrix
\begin{equation*}
    A =
    \begin{pmatrix}
        1 & 0 & 0 & 0 \\
        1 & 1 & 0 & 0 \\
        0 & 0 & 0 & 0 \\
        0 & 0 & 0 & 0
    \end{pmatrix}.
\end{equation*}
Its left and right index pairs for the $(2,2)$ compatible pair are
\begin{align*}
    \{(i_1^k,j_1^k)\}_{k=1}^3 &= \{(1,1),(1,1),(1,1)\}, \\
    \{(i_2^k,j_2^k)\}_{k=1}^3 &= \{(1,1),(2,1),(2,2)\},
\end{align*}
to which correspond the linear indices
\begin{equation*}
    S = \{(l_1^k,l_2^k)\}_{k=1}^3 = \{(1,1),(1,2),(1,4)\}.
\end{equation*}
This set is evidently expressed as the Cartesian product $S = S_1 \times S_2$, where
\begin{equation*}
    S_1 = \{1\} \quad \text{and} \quad S_2 = \{1,2,4\}.
\end{equation*}
Therefore, $A$ admits a $(2,2)$ factorization $A=A_1 \otimes A_2$ and by reverting back to index pairs, we obtain
\begin{equation*}
    I_1=\sparsity(A_1) = \{(1,1)\} \quad \text{and} \quad I_2=\sparsity(A_2) = \{(1,1),(2,1),(2,2)\}.
\end{equation*}
\end{example}
These examples again highlight one of the key properties of our method: it is completely matrix-free and only requires the position of the non-zero entries of $A$. Apart from the Euclidean divisions needed for computing the index pairs, the factorizability of $S$ as a Cartesian product is easily verified by sorting its elements.

\begin{remark}
The procedure described above is closely related to Van Loan's algorithm \cite{van1993approximation}. As a matter of fact, the index set $S$ contains the position of the nonzero entries of $\mathcal{R}(A)$, the so-called rearrangement of $A$ \cite{van1993approximation}. Hence, \Cref{lem: equivalency} is equivalent to verifying that $\mathcal{R}(A)$ is rank $1$. Whereas Van Loan's algorithm explicitly forms $\mathcal{R}(A)$, our method only maps the position of the nonzero entries from $A$ to $\mathcal{R}(A)$, which is all that is needed for verifying its factorizability. In \cite{calderoni2023heuristic}, the authors have instead directly used Van Loan's algorithm for checking the factorizability of $A$.
\end{remark}

Our method easily allows checking whether a matrix is factorizable for a given compatible pair. However, in order to completely describe the Kronecker structure of $A$, \emph{all} possible length $2$ factorizations are needed, which we ensure by iterating over all compatible pairs. Although this may seem prohibitive, the number of compatible pairs is bounded by $\sqrt{n}$, yielding complexity estimates. Running the algorithm over all compatible pairs eventually yields the set of all length $2$ factorizations. For future reference, we denote this set $\dutchcal{F}$ such that
\begin{equation*}
    \dutchcal{F}=\{(n_1,n_2) \colon A \text{ is } (n_1,n_2) \text{ factorizable}\} \subseteq \dutchcal{C}.
\end{equation*}
Note that the order within a pair is important: an $(n_1,n_2)$ factorization may exist independently of an $(n_2,n_1)$ factorization. The next lemma is a cornerstone of our method and explains how to combine length $2$ factorizations to produce factorizations of greater length.

\begin{lemma}
\label{lem: 2fact_2_to_1fact_3}
For $A \in \mathbb{B}^{n \times n}$ and integers $l,p,r>1$, the existence of the following factorizations is equivalent
\begin{equation*}
    \exists
    \begin{cases}
        (l,pr) \\
        (pl,r)
    \end{cases}
    \iff \exists \; (l,p,r).
\end{equation*}
\end{lemma}
\begin{proof}
For the forward implication, assuming there exists two factorizations $(l,pr)$ and $(pl,r)$, then
\begin{align}
    A &= A_1 \otimes A_2, \label{eq: fact_1} \\
    A &= \hat{A}_1 \otimes \hat{A}_2. \label{eq: fact_2}
\end{align}
If $A=0$, the statement trivially holds. If $A \neq 0$, we deduce from \eqref{eq: fact_1} that $A$ is a $l \times l$ block matrix with blocks of size $pr \times pr$ that are either zero or copies of $A_2$. But since $pr$ is a multiple of $r$, the (blockwise) equality of \eqref{eq: fact_1} and \eqref{eq: fact_2} ensures that there exists a matrix $P \in \mathbb{B}^{p \times p}$ such that $A_2 = P \otimes \hat{A}_2$ and consequently, from \eqref{eq: fact_1},
\begin{equation*}
    A = A_1 \otimes A_2 = A_1 \otimes P \otimes \hat{A}_2
\end{equation*}
and there exists an $(l,p,r)$ factorization.

The backward implication immediately follows from the associativity of the Kronecker product: assuming there exists an $(l,p,r)$ factorization, then
\begin{equation*}
    A = A_1 \otimes A_2 \otimes A_3 = A_1 \otimes (A_2 \otimes A_3) = (A_1 \otimes A_2) \otimes A_3
\end{equation*}
and there exist $(l,pr)$ and $(pl,r)$ factorizations.
\end{proof}

\begin{remark}
The statement of \Cref{lem: 2fact_2_to_1fact_3} also holds for real or complex matrices with essentially the same proof arguments.     
\end{remark}

In short, \Cref{lem: 2fact_2_to_1fact_3} shows how to combine two length $2$ factorizations to produce a single length $3$ factorization. More generally, we may combine two factorizations of length $q+1$ and $s+1$ to produce a single factorization of length $q+s+1$. The result is stated in the next corollary, where for a vector $\bm{n}=(n_1,\dots,n_d) \in \mathbb{N}^d$, $\Pi(\bm{n})=\prod_{i=1}^d n_i$.

\begin{corollary}
\label{cor: cor_fact1}
For an integer $p>1$ and integer vectors $\bm{l}=(l_1,\dots,l_q) \in \mathbb{N}^q$, $\bm{r}=(r_1,\dots,r_s) \in \mathbb{N}^s$ all greater than $1$ componentwise, the existence of the following factorizations is equivalent
\begin{equation*}
    \exists
    \begin{cases}
        (\bm{l},p\Pi(\bm{r})) \\
        (p\Pi(\bm{l}),\bm{r})
    \end{cases}
    \iff \exists \; (\bm{l},p,\bm{r}).
\end{equation*}
\end{corollary}
\begin{proof}
If $A$ admits two factorizations $(\bm{l},p\Pi(\bm{r}))$ and $(p\Pi(\bm{l}),\bm{r})$, then
\begin{equation*}
    A = \left(\bigotimes_{i=1}^{q} L_i \right) \otimes R = L \otimes \left(\bigotimes_{i=1}^{s} R_i \right)
\end{equation*}
where the factor matrices $L_i$ and $R_i$ have size $l_i$ and $r_i$, respectively. The proof then follows exactly the same lines as in \Cref{lem: 2fact_2_to_1fact_3} with $A_1 = \bigotimes_{i=1}^{s} L_i$, $A_2 = R$, $\hat{A}_1 = L$ and $\hat{A}_2 = \bigotimes_{i=1}^{s} R_i$.
\end{proof}

The next corollary is the workhorse of our method and allows to combine $m$ factorizations of length $2$ to produce a single factorization of length $m+1$.

\begin{corollary}
\label{cor: cor_fact2}
For integers $l,p_i,r>1$ for $i=1,\dots,q$, the existence of the following factorizations is equivalent
\begin{equation*}
    \exists
    \begin{cases}
        (l, p_1p_2\dots p_q r) \\
        (p_1l, p_2\dots p_q r) \\
        \vdots \\
        (p_1 p_2 \dots p_q l, r)
    \end{cases}
    \iff \exists \; (l,p_1,p_2,\dots,p_q,r).
\end{equation*}
\end{corollary}
\begin{proof}
The proof consists in combining factorizations to gradually increase the length of the product. Considering the first two factorizations of the list and invoking \Cref{cor: cor_fact1} (with $\bm{l}=l$, $p=p_1$ and $\bm{r}=p_2\dots p_q r$), we obtain
\begin{equation*}
    \begin{cases}
        (l, p_1p_2\dots p_q r) \\
        (p_1l, p_2\dots p_q r)
    \end{cases}
    \implies \exists \; (l,p_1,p_2 \dots p_q r).
\end{equation*}
Now combining the result with the third factorization and invoking \Cref{cor: cor_fact1} (with $\bm{l}=(l,p_1)$, $p=p_2$ and $\bm{r}=p_3\dots p_q r$) yields
\begin{equation*}
    \begin{cases}
        (l, p_1, p_2\dots p_q r) \\
        (p_1p_2l, p_3\dots p_q r)
    \end{cases}
    \implies \exists \; (l,p_1,p_2,p_3 \dots p_q r).
\end{equation*}
We continue the process by repeatedly combining the result with the next factorization in the list until we end up with $(l,p_1,p_2,\dots,p_q,r)$. The proof for the backward implication again trivially follows from the associativity of the Kronecker product.
\end{proof}

Based on \Cref{cor: cor_fact2}, we may restrict our attention to length $2$ factorizations to produce factorizations of greater length. \Cref{cor: cor_fact2} also immediately leads to an algorithm where length $2$ factorizations that are related (through the integers $p_i$) are grouped in a single \emph{branch}. In practice, we might end up with multiple branches and a given length $2$ factorization might belong to different branches. The algorithm for constructing all different branches is fairly simple. Starting from the set of all possible length $2$ factorizations
\begin{equation*}
    \dutchcal{F} = \{(l_1,r_1),\dots,(l_s,r_s)\}
\end{equation*}
for some integer $s$, let $\dutchcal{L} = \{l_1,\dots,l_s\}$ and $\dutchcal{R} = \{r_1,\dots,r_s\}$ denote the set of left and right indices, respectively. Without loss of generality, we may assume that $\dutchcal{L}$ is sorted such that $l_1 < l_2 < \dots < l_s$ and the right indices are paired accordingly. Hereafter, we present an algorithm for computing the various branches. Given a set $\dutchcal{X} \subset \mathbb{N}$, we denote $\overline{\dutchcal{X}}$ the reduced set obtained by eliminating all elements that are multiples of other elements in the set. For instance, if $\dutchcal{X}=\{2,3,4,6\}$, then $\overline{\dutchcal{X}}=\{2,3\}$. 

For building the various branches, let $\dutchcal{M}_0=\dutchcal{L}$ and choose an integer $\lcal_0 \in \overline{\dutchcal{M}}_0$. Then, form the set $\dutchcal{M}_1 \subseteq \dutchcal{L}$ of multiples of $\lcal_0$ and choose an integer $\lcal_1 \in \overline{\dutchcal{M}}_1$. Repeat the process until finding an integer $\lcal_q$ that does not have any multiples in $\dutchcal{L}$. The integers $\{\lcal_0,\lcal_1,\dots,\lcal_q\} \subseteq \dutchcal{L}$ are then paired to $\{\rcal_0,\rcal_1,\dots,\rcal_q\} \subseteq \dutchcal{R}$ such that $\lcal_k\rcal_k=n$ for all $k$. Together, they form the left and right indices, respectively, of \emph{one} possible branch. By construction, those indices are all consecutive multiples and there exist integers $\{p_1,\dots,p_q\}$ such that $\lcal_k=p_k \lcal_{k-1}$ for $k=1,\dots,q$ (and similarly for the right indices). By keeping track of all possible choices at each stage of the process, we may construct all possible branches. Assuming that $m$ branches have been constructed and indexing the $k$th branch with the superscript $k$, we eventually obtain
\begin{equation*}
    \begin{cases}
        (\lcal_0^k,\rcal_0^k) =(\lcal_0^k, p_1^kp_2^k\dots p_{q_k}^k \rcal_{q_k}^k) \\
        (\lcal_1^k,\rcal_1^k) =(p_1^k\lcal_0^k, p_2^k\dots p_{q_k}^k \rcal_{q_k}^k) \\
        \vdots \\
        (\lcal_{q_k}^k,\rcal_{q_k}^k) =(p_1^k p_2^k \dots p_{q_k}^k \lcal_0^k, \rcal_{q_k}^k)
    \end{cases}
    \qquad k=1,\dots,m.
\end{equation*}
By convention, we assume that $q_k=0$ corresponds to a branch with the single factorization $(\lcal_0^k,\rcal_0^k)$. Obviously, an analogous construction holds for the right indices and choosing one or the other is just a matter of taste. We will later present a convenient way of visualizing the factorizations resulting from different branches but at this stage an example is opportune. 

\begin{example}
Let us return to \Cref{ex: non_max_mat_12}. Forming the matrix $A$ and computing all length $2$ factorizations, we find 
\begin{equation*}
    \dutchcal{L} = \{2,3,4\}.
\end{equation*}
Following the procedure described earlier we set $\dutchcal{M}_0 = \dutchcal{L}$ and find $\overline{\dutchcal{M}}_0=\{2,3\}$. Thus, there are two possibilities for choosing $\lcal_0$:
\begin{itemize}
    \item If $\lcal_0=2$, $\dutchcal{M}_1=\overline{\dutchcal{M}}_1=\{4\}$ and we must necessarily choose $\lcal_1=4$.
    \item If $\lcal_0=3$ the algorithm ends here since $\dutchcal{L}$ does not contain any multiples of $3$.
\end{itemize}
Consequently, we only have two branches from which we deduce two factorizations
\begin{align*}
\begin{cases}
    (2,6) \\
    (4,3) \\
\end{cases}
&\implies (2,2,3), &
\begin{cases}
    (3,4). \\
\end{cases}
\end{align*}
Those are precisely the two factorizations listed in \Cref{ex: non_max_mat_12}.
\end{example}

The only remaining theoretical question is wether the factorizations resulting from different branches are prime. The next lemma provides a first answer.

\begin{lemma}
\label{lem: primality_1}
For a matrix $A \in \mathbb{B}^{n \times n}$, the right (resp. left) factor of an $(n_1,n_2)$ factorization is decomposable if and only if there exists an $(\hat{n}_1,\hat{n}_2)$ factorization where $\hat{n}_2$ divides $n_2$ (resp. $\hat{n}_1$ divides $n_1$). 
\end{lemma}
\begin{proof}
If the right factor $A_2$ of the $(n_1,n_2)$ factorization is decomposable, then there exist matrices $P$ and $\hat{A}_2$ such that $A_2 = P \otimes \hat{A}_2$. Thus,
\begin{equation*}
    A_1 \otimes A_2 = (A_1 \otimes P) \otimes \hat{A}_2
\end{equation*}
and there exists an $(\hat{n}_1,\hat{n}_2)$ factorization where $\hat{n}_2$ divides $n_2$. 

Now assume there exists an $(\hat{n}_1,\hat{n}_2)$ factorization where $\hat{n}_2$ divides $n_2$. Then, there exists an integer $p$ such that $n_2=p\hat{n}_2$ and therefore $\hat{n}_1=pn_1$. But then, by \Cref{lem: 2fact_2_to_1fact_3},
\begin{equation*}
    \begin{cases}
        (n_1,p\hat{n}_2) \\
        (pn_1,\hat{n}_2)
    \end{cases}
    \implies \exists \; (n_1,p,\hat{n}_2)
\end{equation*}
such that $A=A_1 \otimes P \otimes \hat{A}_2$ and from the uniqueness of the factorization (\Cref{lem: uniqueness}) we deduce that $A_2=P \otimes \hat{A}_2$ and $A_2$ is decomposable. The proof for the left factor follows similar arguments.
\end{proof}

Assume we have constructed a branch 
\begin{equation*}
    \begin{cases}
        (l, p_1p_2\dots p_q r) \\
        (p_1l, p_2\dots p_q r) \\
        \vdots \\
        (p_1 p_2 \dots p_q l, r)
    \end{cases}
\end{equation*}
according to the procedure described earlier. We will now show that the $(l,p_1,p_2,\dots,p_q,r)$ factorization obtained by combining all length $2$ factorizations within the branch is prime provided $\dutchcal{F}$ lists \emph{all} possible length $2$ factorizations.

\begin{corollary}
\label{cor: primality_2}
If $\dutchcal{F}$ contains all possible length $2$ factorizations, then the factorizations constructed from each separate branch are prime.    
\end{corollary}
\begin{proof}
We will first prove that the leftmost and rightmost factors of the $(l,p_1,p_2,\dots,p_q,r)$ factorization are prime. Indeed, by construction, $l$ does not have any divisor in $\dutchcal{L}$ so by \Cref{lem: primality_1}, the leftmost factor is prime. Moreover, assume by contradiction that the rightmost factor is decomposable such that $r=\tilde{p}\tilde{r}$. But then $\dutchcal{F}$ must contain the $(\tilde{p}p_1 p_2 \dots p_q l, \tilde{r})$ factorization and consequently there would exist a multiple of $p_1 p_2 \dots p_q l$ in $\dutchcal{L}$. This contradicts the termination criterion for constructing a branch and consequently the rightmost factor must also be prime. To prove that all the inner factors of the decomposition are also prime, assume by contradiction that one of the inner factors, say $P_i$, is decomposable such that $p_i=\tilde{l}_i\tilde{r}_{i}$. Therefore, $\dutchcal{F}$ must contain the $(p_1p_2\dots p_{i -1}\tilde{l}_il, \tilde{r}_{i}p_{i+1}\dots p_q r)$ factorization. But then there would exist a multiple of $\ell_{i-1}=p_1p_2\dots p_{i-1} l$ in $\dutchcal{L}$ that divides $\ell_i=p_1p_2\dots p_i l$. Once again, this clashes with the construction of a branch since any left index $\ell$ that is a multiple of $\ell_{i-1}$ cannot be a divisor of $\ell_i$. Thus, we conclude that all internal factors are also prime.
\end{proof}

\Cref{cor: primality_2} proves that distinct branches lead to distinct prime factorizations and remarkably, we may reconstruct those factorizations from the enumeration of all length $2$ factorizations. However, the primality guarantee is lost if $\dutchcal{F}$ only \emph{partially} lists length $2$ factorizations. This does not prevent us from applying the algorithm but the factor matrices in the resulting decompositions may still be decomposable. Also bear in mind that even if a factorization is prime, the sizes of the factor matrices in the decomposition are generally not prime numbers (see \Cref{ex: non_max_mat_12}). Nevertheless, maximal matrices are again special in this regard and their properties are summarized in the next theorem.

\begin{theorem}
\label{th: maximal_mat}
Let $A \in \mathbb{B}^{n \times n}$ be a decomposable maximal matrix and let $n = \prod_{j=1}^m p_j^{k_j}$ denote the prime decomposition of $n$. Then, for an $(n_1,\dots,n_\ell)$ factorization resulting from a given branch,
\begin{enumerate}
    \item $n_i \in \{p_1,\dots,p_m\}$ for all $i=1,\dots,\ell$.
    \item $\ell = \sum_{j=1}^m k_j$.
    \item The number of branches is $\frac{\ell !}{\prod_{j=1}^m k_j!}$.
\end{enumerate}
\end{theorem}
\begin{proof}
Since $A$ is maximal, $\dutchcal{F}=\dutchcal{C}$ and $\dutchcal{L}$ contains all non-trivial divisors of $n$. Thus,
\begin{equation*}
    \dutchcal{L} = \left\{\prod_{j=1}^m p_j^{q_j} \colon 0 < \sum_{j=1}^m q_j < \sum_{j=1}^m k_j,\ 0 \leq q_j \leq k_j,\ j=1,\dots,m \right\}.
\end{equation*}
From our algorithmic construction, the sizes $n_1,\dots,n_\ell$ resulting from any branch form a sequence of prime numbers, where each $p_j$ appears $k_j$ times among $n_1,\dots,n_\ell$. Consequently, the length of any such factorization is
\begin{equation*}
    \ell = \sum_{j=1}^m k_j.
\end{equation*} 
Moreover, the number of distinct sequences (or branches) is the number of permutations of $\ell$ objects, among which $k_1$ are indistinguishable, $k_2$ are indistinguishable and so on. From elementary statistics, this number is
\begin{equation*}
    \frac{\ell !}{\prod_{j=1}^m k_j!}.
\end{equation*}
\end{proof}

\Cref{th: maximal_mat} also implies that for a general $A \in \mathbb{B}^{n \times n}$, the length of any factorization cannot exceed $\sum_{j=1}^m k_j$, where the $k_j$ are the exponents in the prime decomposition of $n$. This result merely confirms what could be straightforwardly inferred without setting up any heavy machinery. The results of \Cref{th: maximal_mat} are easily verified on \Cref{ex: max_mat_12}. For completeness, we present another more descriptive example by applying the algorithm outlined earlier.

\begin{example}
\label{ex: max_mat_24}
Let
\begin{equation*}
    \dutchcal{L}=\{2,3,4,6,8,12\}
\end{equation*}
be the (sorted) left indices of all possible length $2$ factorizations of a maximal matrix of size $24=2^3 \cdot 3$. Thus, $m=2$ with prime numbers $p_1=2$, $p_2=3$ and integers $k_1=3$, $k_2=1$. Let us first apply our algorithm while ignoring the fact that $A$ is maximal. As explained earlier we initialize $\dutchcal{M}_0=\dutchcal{L}$ and choose an integer in $\overline{\dutchcal{M}}_0=\{2,3\}$. The integers in $\overline{\dutchcal{M}}_0$ are the roots of any factorization we may form. Keeping track of the various choices at each stage of the process exhibits the various branches, leading to four distinct prime factorizations.
\begin{itemize}
    \item If $\lcal_0=2$, $\dutchcal{M}_1 = \{4,6,8,12\}$ and $\overline{\dutchcal{M}}_1 = \{4,6\}$
    \begin{itemize}
        \item If $\lcal_1=4$, $\dutchcal{M}_2=\{8,12\}$ and $\overline{\dutchcal{M}}_2=\{8,12\}$.
        \begin{itemize}
            \item If $\lcal_2=8$ the branch ends since $\dutchcal{L}$ does not contain any multiple of $8$.
            \item If $\lcal_2=12$ the branch ends since $\dutchcal{L}$ does not contain any multiple of $12$.
        \end{itemize}
        \item If $\lcal_1=6$, $\dutchcal{M}_2=\{12\}$ and $\overline{\dutchcal{M}}_2=\{12\}$. Consequently, $\lcal_2=12$ and the branch ends.
    \end{itemize}
    \item If $\lcal_0=3$, $\dutchcal{M}_1 = \{6,12\}$, $\overline{\dutchcal{M}}_1 = \{6\}$ and the only one possible path is $\lcal_1=6$ and $\lcal_2=12$.
\end{itemize}

We have therefore found four different branches
\begin{align*}
\begin{cases}
    (2,12) \\
    (4,6) \\
    (8,3)\\
\end{cases}
&\implies (2,2,2,3), &
\begin{cases}
    (2,12) \\
    (4,6) \\
    (12,2)\\
\end{cases}
&\implies (2,2,3,2), \\
\begin{cases}
    (2,12) \\
    (6,4) \\
    (12,2)\\
\end{cases}
&\implies (2,3,2,2), &
\begin{cases}
    (3,8) \\
    (6,4) \\
    (12,2)\\
\end{cases}
&\implies (3,2,2,2).    
\end{align*}
In agreement with \Cref{th: maximal_mat}, the sizes of the factor matrices in each factorization are prime numbers, the length of each factorization is $\ell = k_1 + k_2 = 4$ and the number of factorizations is $\ell!/(k_1!k_2!)=4$.
\end{example}

However, it is worthwhile noting that the Kronecker product of two maximal matrices is generally not maximal, as shown in the next example.

\begin{example}
\label{ex: mat_36}
Consider the maximal matrix
\begin{align*}
A 
&=
\begin{pmatrix}
1 & 0 \\
1 & 0
\end{pmatrix}
\otimes
\begin{pmatrix}
1 & 0 & 1 \\
1 & 0 & 1 \\
1 & 0 & 1
\end{pmatrix} \\
&=
\begin{pmatrix}
1 & 1 & 0 \\
1 & 1 & 0 \\
1 & 1 & 0
\end{pmatrix}
\otimes
\begin{pmatrix}
1 & 0 \\
1 & 0
\end{pmatrix}.
\end{align*}
Taking the Kronecker product of $A$ with itself, we obtain a matrix of size $36=2^2 \cdot 3^2$. However, it only admits $5$ prime factorizations, which is smaller than the $6$ factorizations it could potentially have. Indeed, the $(2,2,3,3)$ factorization is missing because $4$ does not belong to the set of left indices.
\end{example}

In addition to the sizes of the factor matrices, one may also be interested in their sparsity pattern. Fortunately, once the sizes of the factor matrices in a decomposition are known, one may also easily retrieve their sparsity pattern as a simple post-processing step. This operation only requires labeling the length $2$ factorizations constituting a branch and storing the sparsity pattern of their factors. Note that the sparsity pattern of length $2$ factorizations is an immediate byproduct of the factorizability check and does not entail any extra cost. The sparsity pattern of factorizations of length $\ell > 2$ is then deduced one factor at a time by combining the sparsity pattern of consecutive length $2$ factorizations. Indeed, the length $2$ factorizations entering a branch are obtained by simply regrouping an increasingly large number of left factors together. More specifically, the length $2$ factorizations constituting the branch of an arbitrary $(l,p_1,p_2,\dots,p_q,r)$ factorization $L \otimes P_1 \otimes P_2 \otimes \dots \otimes P_q \otimes R$ correspond to the following products:
\begin{equation*}
    \begin{cases}
        (l, p_1p_2\dots p_q r) \qquad L \otimes (P_1 \otimes P_2 \otimes \dots \otimes P_q \otimes R),\\
        (p_1l, p_2\dots p_q r) \qquad (L \otimes P_1) \otimes (P_2 \otimes \dots \otimes P_q \otimes R),\\
        \vdots \\
        (p_1 p_2 \dots p_q l, r) \qquad (L \otimes P_1 \otimes P_2 \otimes \dots \otimes P_q) \otimes R.
    \end{cases}
\end{equation*}
Clearly, the sparsity pattern of the leftmost factor $L$ is simply the sparsity pattern of the left factor of the first length $2$ decomposition constituting the branch. The sparsity pattern of $P_1$ is then deduced from the sparsity patterns of $L$ and $L \otimes P_1$, the latter simply being the sparsity pattern of the left factor in the second length $2$ factorization. More generally, one deduces the sparsity pattern of $P_{i}$ from the sparsity patterns of $L \otimes P_1 \otimes \dots \otimes P_{i-1}$ and $L \otimes P_1 \otimes \dots \otimes P_{i-1} \otimes P_{i}$, which are nothing but the left factors of the $i$th and $(i+1)$th factorizations entering the branch. Finally, the sparsity pattern of the rightmost factor $R$ is simply the sparsity pattern of the right factor of the last factorization. In summary, this process consists in repeatedly finding the sparsity pattern of the right factor of a Kronecker product knowing the sparsity pattern of the left factor and the one of the product. The solution to this problem is straightforward: assuming we are given an $(n_1,n_2)$ factorization of $A = A_1 \otimes A_2$, where the sparsity patterns of $A$ and $A_1$ are known, then the sparsity pattern of $A_2$ is simply read from the $(i,j)$th block of $A$ of size $n_2$ corresponding to any nonzero entry $(i,j)$ of $A_1$. Similarly, if instead the sparsity patterns of $A$ and $A_2$ were known, the sparsity pattern of $A_1$ would easily be deduced by tracking down the position of the nonzero blocks of $A$ of size $n_2$. Hence, one could equivalently apply the aforementioned process from right to left instead of from left to right. In either case, all that is needed is the sparsity pattern of the factor matrices for the length $2$ factorizations entering a branch.

Now that we have a convenient way of constructing prime factorizations and finding the sparsity pattern of its factors, we present in the next section an elegant way of visualizing them.

\section{The decomposition graph}
\label{se: decomposition_graph}
In this section, we build a directed graph $G$ that encodes the construction of factorization branches. A graph is characterized by a set of vertices $V(G)$ and a set of edges $E(G)$, that are pairs of vertices. In our setting, this graph is more specifically a multigraph: it may have multiple edges connecting the same pair of vertices. Its construction obeys the following rules:
\begin{itemize}[noitemsep]
    \item $V(G)=\dutchcal{L}$.
    \item For each branch $k$ we connect successive left indices by an edge such that $\dutchcal{L}^k = \{\lcal_0^k,\lcal_1^k,\dots,\lcal_{q_k}^k\} \subseteq \dutchcal{L}$ forms a path. Each path is identified by a distinct color or label.
    \item For the $k$th path $\dutchcal{L}^k$, the ``weights'' on each edge are the integers $p_j^k$.
\end{itemize}

From this construction, we immediately deduce the following properties:
\begin{itemize}[noitemsep]
    \item The length of the $k$th factorization $\ell_k = |\dutchcal{L}^k|+1$ is related to the length of the $k$th path.
    \item Isolated vertices correspond to prime length $2$ factorizations (but are still identified as ``paths'').
    \item The number of paths is equal to the number of branches, which is itself equal to the number of prime factorizations.
\end{itemize}

For reading the sizes of the factor matrices in the $k$th factorization, we start from the root $\lcal_0^k$ and read the successive weights on the edges until arriving at the end of the path. The size of the last factor matrix is then $n/\lcal_{q_k}^k$. The decomposition graphs for the various examples we have encountered so far are shown in \Cref{fig: decomp_graphs_examples}. It unveils in a single figure the Kronecker structure of a matrix and becomes especially useful for matrices admitting multiple distinct factorizations. Furthermore, since it visualizes the structure rather than the matrix, different matrices may have the same decomposition graph.

\begin{figure}[H]
     \centering
     \begin{subfigure}[t]{0.48\textwidth}
    \centering
    \includegraphics[width=\textwidth]{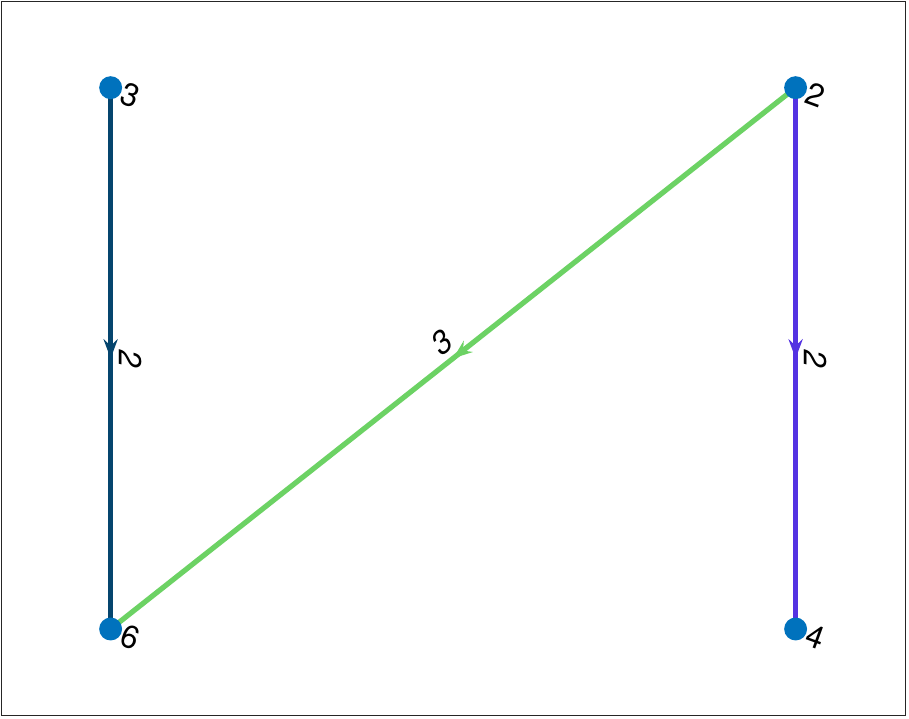}
    \caption{\Cref{ex: max_mat_12}}
    \label{fig: maximal_mat_12}
     \end{subfigure}
     \hfill
     \begin{subfigure}[t]{0.48\textwidth}
    \centering
    \includegraphics[width=\textwidth]{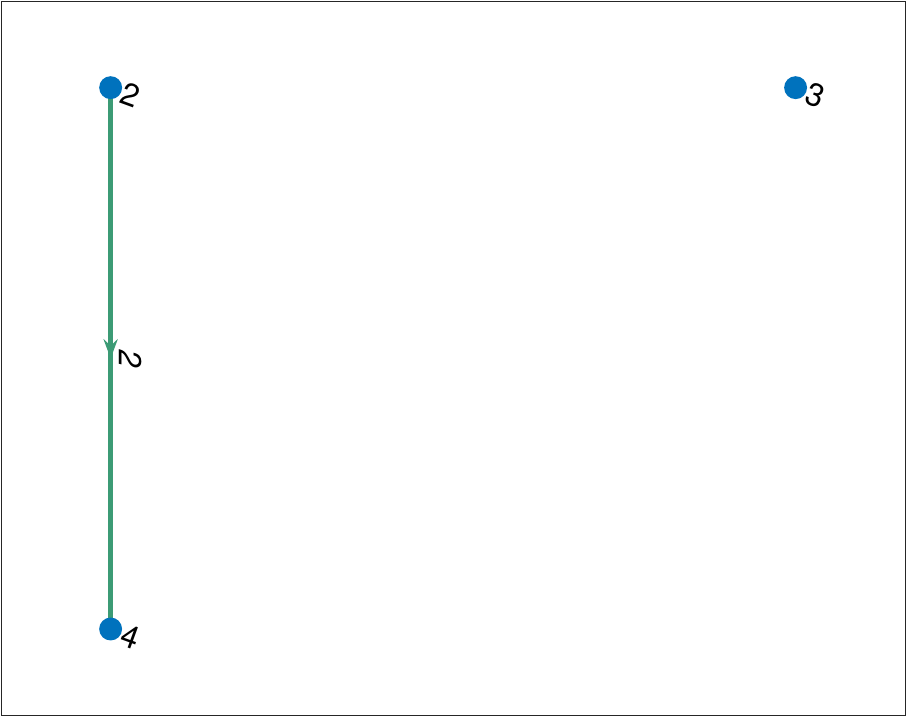}
    \caption{\Cref{ex: non_max_mat_12}}
    \label{fig: non_max_mat_12}
     \end{subfigure}
     \hfill
    \begin{subfigure}[t]{0.48\textwidth}
    \centering
    \includegraphics[width=\textwidth]{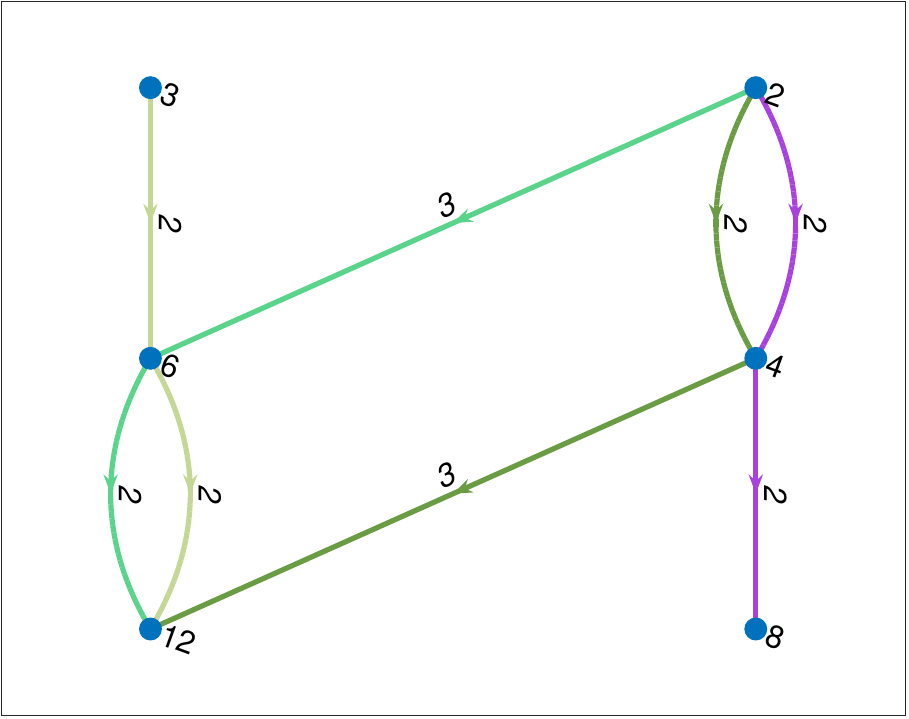}
    \caption{\Cref{ex: max_mat_24}}
    \label{fig: max_mat_24}
     \end{subfigure}
     \hfill
    \begin{subfigure}[t]{0.48\textwidth}
    \centering
    \includegraphics[width=\textwidth]{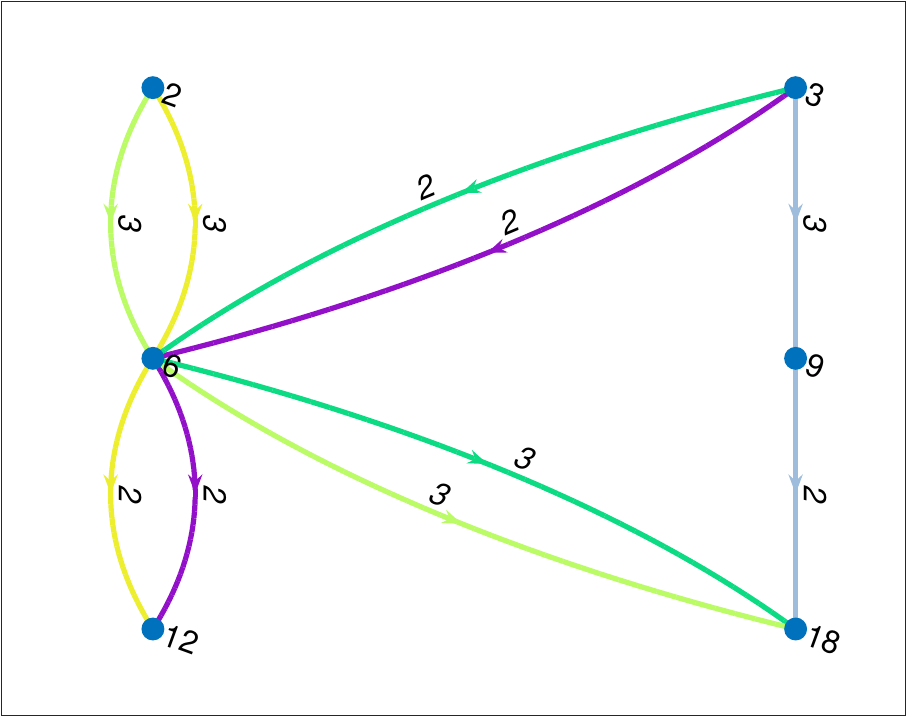}
    \caption{\Cref{ex: mat_36}}
    \label{fig: mat_36}
     \end{subfigure}
     \hfill
    \caption{Decomposition graphs}
    \label{fig: decomp_graphs_examples}
\end{figure}

\section{Applications}
\label{se: applications}
\subsection{Space-time isogeometric discretizations}
\label{se: space_time_iga}
One of the most significant applications of this research consists in identifying potentially good (approximate) Kronecker factorizations of a real or complex matrix $B$ just by analyzing its sparsity pattern. As an instructive example, we consider a space-time isogeometric discretization of the heat equation, a well studied parabolic PDE. For time-dependent PDEs, such as the heat or the wave equation, space-time methods discretize both spatial and temporal domains with finite elements \cite{loli2020efficient,loli2023high}. In isogeometric analysis, finite element spaces are tensor products of smooth spline spaces \cite{hughes2005isogeometric,cottrell2009isogeometric} and lead to large structured system matrices where the spatial and temporal degrees of freedom are solved all at once \cite{mcdonald2018preconditioning}. We consider in this section the discretization of the heat equation
\begin{align*}
 \partial_{t} u(\mathbf{x},t)-\Delta u(\mathbf{x},t) &=f(\mathbf{x},t) & &\text{ in } \Omega \times (0,T],  \\
 u(\mathbf{x},t)&=0 & &\text{ on } \partial \Omega \times (0,T],  \\
 u(\mathbf{x},0)&=u_0(\mathbf{x}) & &\text{ in } \Omega,
\end{align*}
over the space-time cylinder $Q=\Omega \times (0,T)$, where $\Omega \subset \mathbb{R}^3$ is the magnet-shaped domain shown in \Cref{fig: magnet} and $T>0$ is the final time. The unknown function $u \colon \overline{Q} \to \mathbb{R}$ represents a temperature field, $f$ is a known source term and $u_0$ is an initial condition. 

\begin{figure}[H]
    \centering
    \includegraphics[width=0.5\linewidth]{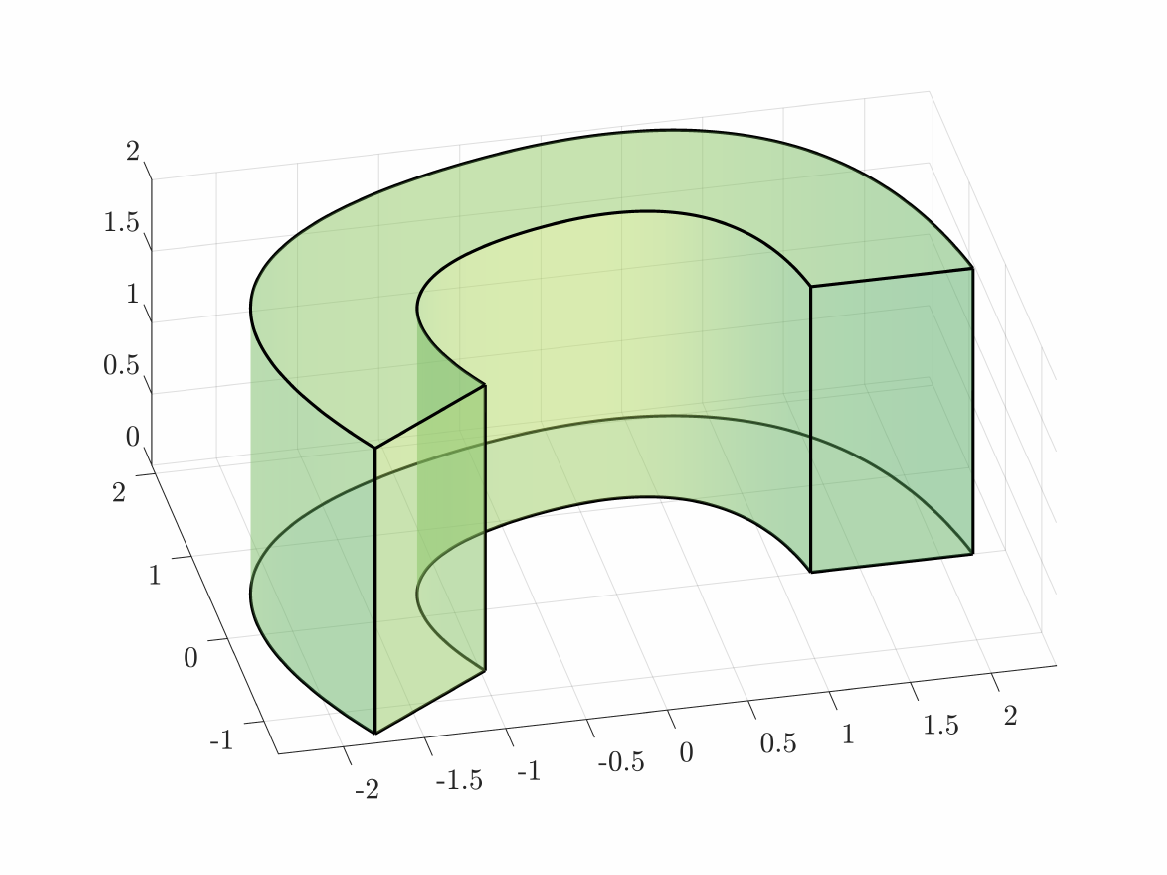}
    \caption{Magnet shaped domain taken from \cite[Figure 5.7]{voet2025mass}}
    \label{fig: magnet}
\end{figure}

Discretizing the problem with space-time isogeometric methods requires solving a large linear system whose coefficient matrix is (see e.g. \cite{loli2020efficient})
\begin{equation}
\label{eq: B_heat_equation}
B = W_t \otimes M_s + M_t \otimes K_s
\end{equation}
where 
\begin{align*}
    (W_t)_{ij} &= \int_0^T b_j'(t)b_i(t) \mathrm{d}t & (M_t)_{ij} &= \int_0^T b_j(t)b_i(t) \mathrm{d}t \qquad i,j=1,\dots,n_t\\
\end{align*}
are the ``temporal'' finite element matrices and $\{b_i(t)\}_{i=1}^{n_t}$ are the B-spline basis functions for the temporal domain $(0,T)$. Similarly,
\begin{align*}
    (K_s)_{ij} &= \int_\Omega \nabla B_j(\bm{x}) \cdot \nabla B_i(\bm{x}) \mathrm{d}\Omega & (M_s)_{ij} &= \int_\Omega B_j(\bm{x})B_i(\bm{x}) \mathrm{d}\Omega \qquad i,j=1,\dots,n_s\\
\end{align*}
are the ``spatial'' stiffness and mass matrices and $\{B_i(t)\}_{i=1}^{n_s}$ are the B-spline basis functions for the spatial domain $\Omega$. The interested reader may consult \cite{hughes2005isogeometric,cottrell2009isogeometric} for the construction of the B-spline basis and a gentle introduction to isogeometric analysis. Although the spatial stiffness and mass matrices $K_s$ and $M_s$ are generally not a Kronecker product, for certain spline parameterizations they are often exceedingly well approximated by a sum of Kronecker products \cite{hofreither2018black,mantzaflaris2017low,scholz2018partial}, a property that is also reflected in their sparsity pattern. Indeed, these matrices exhibit a hierarchical block structure, where the block sizes and their bandwidths are directly related to the spline discretization parameters \cite{hofreither2018black,voet2025mass}. Their Kronecker product with the banded temporal stiffness and mass matrices $W_t$ and $M_t$ in \eqref{eq: B_heat_equation} just adds another hierarchical level, as shown in \Cref{fig: sparsity_application_1}. Thus, the sparsity pattern of $B$ still hides a Kronecker product although the matrix itself may not be exactly factorized. In this example, we apply our algorithm to identify the sizes of the constituting factor matrices before computing an approximate Kronecker factorization. This example is obviously contrived since the sizes are commonly deduced from the discretization parameters and boundary conditions. However, that information may not always be available if the discretization is carried out independently of the linear system solver. In contrast, our algorithm operates in a ``black-box'' manner and allows reading the discretization parameters just by analyzing the sparsity pattern of the matrix. As a matter of fact, the discretization parameters for this example lead to a system matrix $B$ of size $n=53568$ and our algorithm perfectly recovers the (unique) $(31,12,12,12)$ prime factorization of its sparsity pattern, in agreement with \Cref{fig: sparsity_application_1}\footnote{The code for reproducing the results is freely available at the following address: \url{https://github.com/YannisVoet/Kronecker}}. Once the sizes are known, one may easily deduce the sparsity pattern of the factors constituting the decomposition, as well as their bandwidth. In this example, each factor is banded of bandwidth $2$. We may relate all that information back to the mesh sizes and spline orders of the discretization. However, in this example, we are primarily interested in computing an approximate Kronecker factorization of $B$, which may later serve as a preconditioner. The Kronecker structure already exhibited in \eqref{eq: B_heat_equation} suggests approximating $B$ by a sum of two length $4$ Kronecker products
\begin{equation*}
    B \approx \tilde{B} = A_1 \otimes B_1 \otimes C_1 \otimes D_1 + A_2 \otimes B_2 \otimes C_2 \otimes D_2.
\end{equation*}
Computing this approximation requires tensor decomposition techniques, as described in \cite{langville2004akronecker}. For this purpose, we have used Matlab's Tensor Toolbox \cite{bader2023matlab}. The sizes of the factor matrices entering the decomposition are directly supplied by our algorithm. Those sizes directly reflect the structure of the matrix and are a rather natural choice for attempting an approximate factorization. Of course, there exist multiple other candidate sizes for computing approximate factorizations but the large error incurred may later impede on the preconditioner's effectiveness. For demonstrating it, let us compute an approximate factorization for three randomly chosen, albeit compatible sizes $\bm{n}_i$ for $i=1,2,3$ and compare them to the sparsity-informed guess $\bm{n}=(31,12,12,12)$. As shown in \Cref{tab: errors}, the approximation errors in the Frobenius norm are about $10$ times larger than for the sparsity-informed guess. Approximating $K_s$ and $M_s$ each with a single length $3$ factorization and plugging the result in \eqref{eq: B_heat_equation} also produces a valid approximation but the error $(0.2855)$ is still significantly larger than for the sparsity-informed length $4$ factorization. Smaller approximation errors are expected to produce better preconditioners. For illustrating it, we solve the linear system $B\bm{x}=\bm{e}_1$, where $\bm{e}_1$ is the first vector of the canonical basis of $\mathbb{R}^n$. In this experiment, the nonsymmetric linear system is solved iteratively with a right preconditioned GMRES method \cite{saad2003iterative}, restarted every 30 iterations until reaching an absolute residual norm of $10^{-8}$. The method converged after $109$, $13$ and $19$ iterations for $\bm{n}_1$, $\bm{n}_2$ and $\bm{n}_3$ whereas only $8$ iterations were required for $\bm{n}$. The iteration count increases to $52$ when separately approximating $K_s$ and $M_s$. This example further highlights the value of approximate Kronecker factorizations for building preconditioners and is in this context one among many other strategies proposed for space-time discretizations; see e.g. \cite{loli2020efficient,mcdonald2018preconditioning,kressner2023improved}.

\begin{figure}[H]
     \centering
     \begin{subfigure}[t]{0.24\textwidth}
    \centering
    \includegraphics[width=\textwidth]{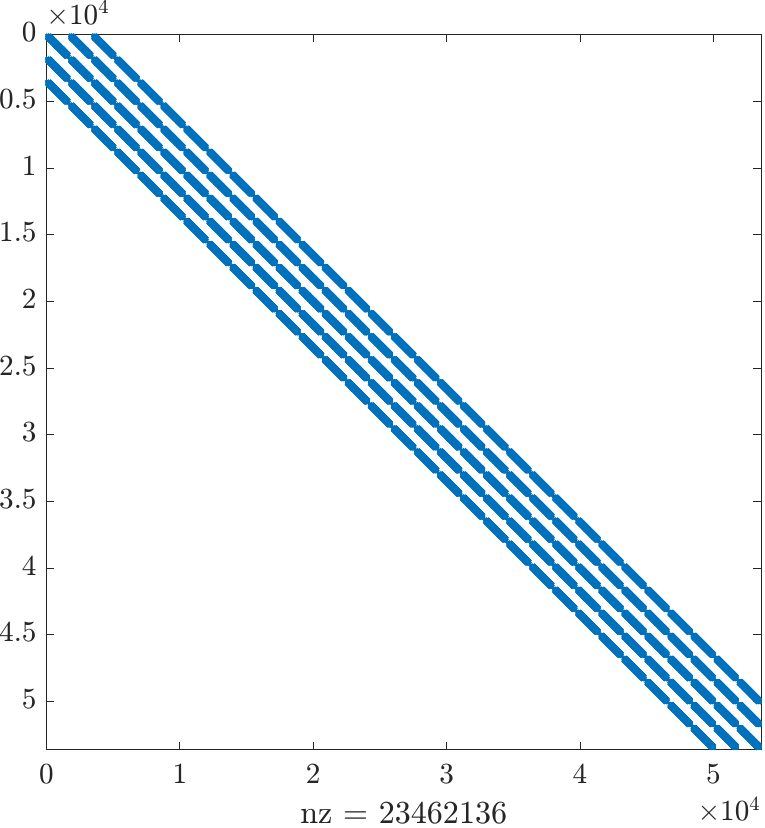}
    \caption{Sparsity pattern of $B$}
    \label{fig: sparsity_1}
     \end{subfigure}
     \hfill
     \begin{subfigure}[t]{0.24\textwidth}
    \centering
    \includegraphics[width=\linewidth]{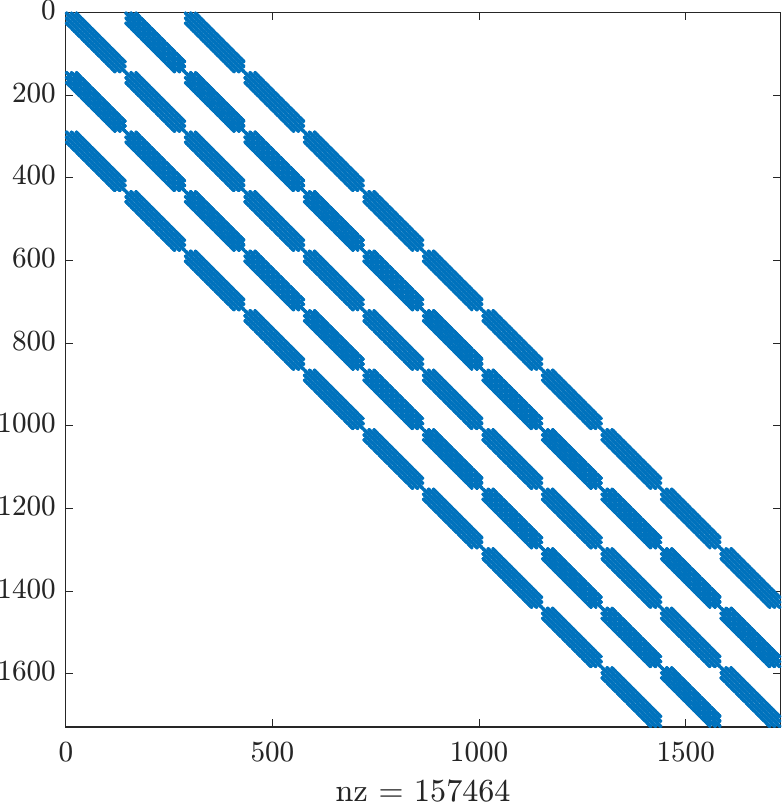}
    \caption{Sparsity pattern of the leading block in \Cref{fig: sparsity_1}}
    \label{fig: sparsity_2}
     \end{subfigure}
     \hfill
    \begin{subfigure}[t]{0.24\textwidth}
    \centering
    \includegraphics[width=\linewidth]{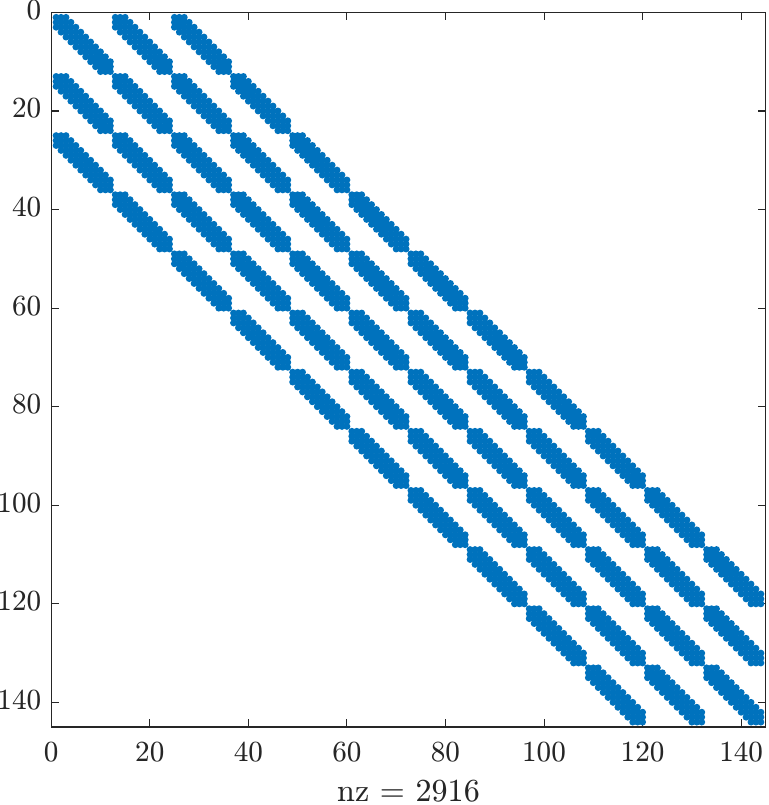}
    \caption{Sparsity pattern of the leading block in \Cref{fig: sparsity_2}}
    \label{fig: sparsity_3}
     \end{subfigure}
     \hfill
    \begin{subfigure}[t]{0.24\textwidth}
    \centering
    \includegraphics[width=\linewidth]{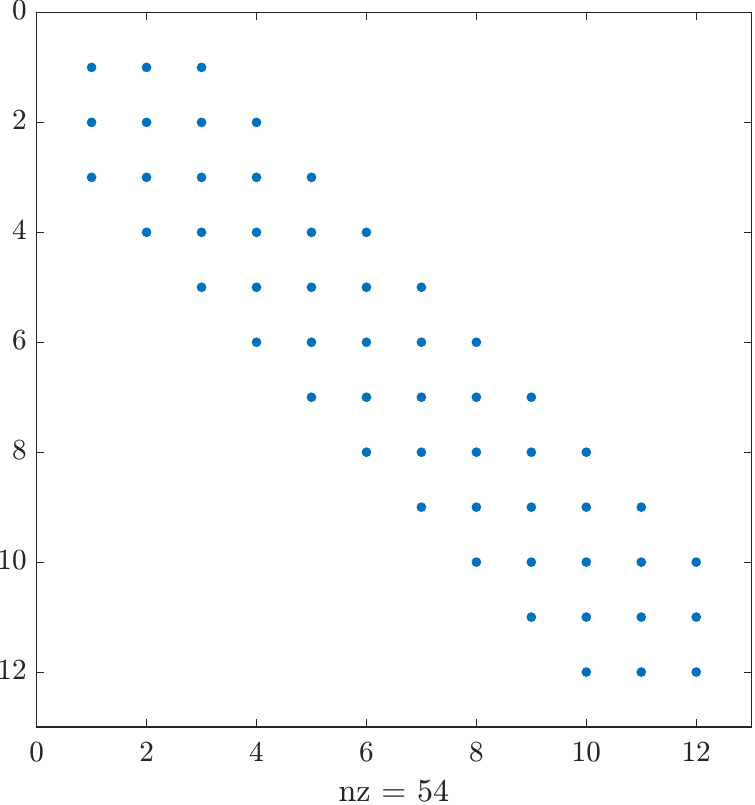}
    \caption{Sparsity pattern of the leading block in \Cref{fig: sparsity_3}}
    \label{fig: sparsity_4}
     \end{subfigure}
     \hfill
    \caption{Hierarchical block structure of $B$}
    \label{fig: sparsity_application_1}
\end{figure}

\begin{table}[H]
    \centering
    \begin{tabular}{|m{3cm}|m{2cm}|}
    \hline
     Size & $\|B-\tilde{B}\|_F$ \\
     \hline
     $\bm{n}=(31,12,12,12)$ & 0.0489 \\
     $\bm{n}_1=(31,144,6,2)$ & 0.5561 \\
     $\bm{n}_2=(31,12,6,24)$ & 0.3832 \\
     $\bm{n}_3=(62,6,6,24)$ & 0.4973 \\
    \hline
    \end{tabular}
    \caption{Sizes of the factor matrices and resulting approximation errors}
    \label{tab: errors}
\end{table}

\begin{remark}
In the example of \Cref{se: space_time_iga}, the sparsity pattern had a unique prime factorization. In fact, from the point of view of preconditioning, the existence of multiple prime factorizations is undesirable since it still leaves multiple candidate sizes. In the extreme case of dense matrices, the sparsity pattern is maximal and does not help at all for choosing candidate sizes. In such cases, exhaustive search strategies described in \cite{cai2022kopa} might become necessary, unless additional information on the problem’s origin is known.    
\end{remark}

\subsection{Decomposition and visualization of Kronecker graphs}
Our second application pertains to the representation and visualization of networks as Kronecker graphs. Given two graphs $G_1$ and $G_2$ with vertex set $V(G_i)$ and edge set $E(G_i)$ for $i=1,2$, their Kronecker product $G_1 \otimes G_2$ (sometimes also called direct, cardinal, categorical or tensor product) is a larger graph with vertex set defined as the Cartesian product 
\begin{equation*}
    V(G) = V(G_1) \times V(G_2) = \{(u,v) \colon u \in V(G_1),\, v \in V(G_2)\}
\end{equation*}
and edge set
\begin{equation*}
     E(G) = \{((u,v),(u',v')) \colon (u,u') \in E(G_1),\, (v,v') \in V(G_2)\}.
\end{equation*}
Up to a relabeling of the vertices, the adjacency matrix $A$ of a Kronecker graph is the Kronecker product of the adjacency matrices $A_i$ \cite{calderoni2021direct}; i.e. there exists a permutation matrix $P$ such that
\begin{equation*}
    P^TAP = A_1 \otimes A_2.
\end{equation*}
This permutation matrix is a major hurdle for identifying Kronecker graphs and a key difference with the matrix case. Nevertheless, Kronecker graphs are attractive for multiple reasons. Firstly, the structure and properties of Kronecker graphs are easier to analyze and often deduced from the properties of the smaller graphs constituting them. Thus, they effectively condense information. Secondly, many of those properties accurately model real networks. For instance, in the machine learning community, Leskovec et al. \cite{leskovec2007scalable,leskovec2010kronecker} realized that synthetic Kronecker graphs could somewhat reproduce the degree distribution, diameter and spectrum of real networks, a set of properties earlier models often failed to mimic. Finally, the inherent structure of Kronecker graphs also helps visualizing them. Here we will especially focus on this last point, which is largely independent of any application. Despite a late surge of interest in Kronecker graphs, attempts at visualizing them are rather sparse. As a matter of fact, \cite[Chapter 11]{kepner2011graph} is the only noteworthy contribution we are aware of. Therein, apart from visualizing the sparsity pattern of the adjacency matrix, the authors propose a 3D visualization of Kronecker graphs by projection them onto the surface of a sphere. However, the figures produced are somewhat confusing and difficult to interpret. In this section, we present a new visualization technique that exploits the Kronecker structure of the (reordered) adjacency matrix $A$. Assuming we have found an $(n_1,n_2,\dots,n_d)$ factorization of $A$, we first label the vertices and identify them with tuples $\bm{v}=(v_1,v_2,\dots,v_d)$, where $1 \leq v_i \leq n_i$ for $i=1,\dots,d$. This mindset allows modeling short, medium or long distance interactions between vertices belonging, say, to different communities. If only the last few indices of two connected vertices $\bm{v}_1$ and $\bm{v}_2$ differ, then their link represents relatively short-distance interactions, within the same community. On the contrary, if any of the first few indices differ, then they might represent long-distance inter-community interactions. The construction described below conforms with such intuitive understanding of the network by explicitly defining the position of the graph's vertices. In order to model communities on different levels (e.g. local, regional, continental, inter-continental, etc), we first draw a circle of unit radius $r_1$ centered at the origin of the complex plane and uniformly place $n_1$ nodes along the circle. These nodes are simply the $n_1$ roots of unity $z_k = \mathrm{e}^{\frac{2 \pi (k-1) \mathrm{i}}{n_1}}$ for $k=1,\dots,n_1$ and model the centers of long-distance (inter-continental) communities. For moving down to the continental scale, we draw circles around each point $z_k$ with a smaller radius $r_2$ and place $n_2$ nodes along each of these circle for modeling the centers of continental hubs. We then repeat the process until reaching the local scale. At this stage, the points placed on the smallest circles are the position of the vertices of the graph. The construction process is illustrated in \Cref{fig: skeleton} for $d=3$ and $(n_1,n_2,n_3) = (4,3,2)$. Note that it only depends on the sizes of the factor matrices, not on the connectivity of the network.

The algorithm combines rotations and translations to produce a visually pleasing and interpretable representation of the network. More formally, for each level $j=1,\dots,d$, we first place equidistance points on a circle in the complex plane
\begin{equation*}
    z^{(j)}_{k} = r_j \mathrm{e}^{\frac{2 \pi (k-1) \mathrm{i}}{n_j}} \qquad k=1,\dots, n_j, \quad j=1,\dots,d.
\end{equation*}
where $(r_j)_{j=1}^n$ is a decreasing sequence of radii and $\mathrm{i}$ denotes the imaginary unit. Additionally, we denote
\begin{equation*}
    \theta^{(j)}_{k} = \arg(z^{(j)}_{k}) + \frac{\pi}{2} \quad \text{and} \quad r^{(j)}_{k} = \mathrm{e}^{\theta^{(j)}_{k} i}
\end{equation*}
the (shifted) phase angle and rotation, respectively. The shift of $\frac{\pi}{2}$ is purely for aesthetic reasons as it produced clearer figures. The position of the vertices in the graph is then defined recursively, starting from the bottom of the hierarchy, moving upwards and augmenting the index set at each step. We first initialize $g^{(d)}_{k_d} = z^{(d)}_{k_d}$. Then, assuming $g^{(j+1)}_{(k_{j+1},\dots,k_d)}$ is known, we define
\begin{equation*}
    g^{(j)}_{(k_j,k_{j+1},\dots,k_d)} = z^{(j)}_{k_j} + r^{(j)}_{k_j} g^{(j+1)}_{(k_{j+1},\dots,k_d)} \quad j=1,\dots,d-1.
\end{equation*}
At the end of the recursion, $p_{\bm{k}} = g^{(1)}_{\bm{k}}$ defines the position of vertex $\bm{k}=(k_1,k_2,\dots,k_d)$ and its coordinate values are finally retrieved as
\begin{equation*}
    x_{\bm{k}} = \re(p_{\bm{k}}), \quad y_{\bm{k}} = \im(p_{\bm{k}}).
\end{equation*}
Let us visualize the output of this algorithm on the Kronecker graph whose adjacency matrix is
\begin{equation}
\label{eq: adjacency_1}
A = A_1 \otimes A_2 \otimes A_3 =
\begin{pmatrix}
1 & 0 & 0 & 0 \\
1 & 1 & 0 & 0 \\
1 & 1 & 1 & 0 \\
0 & 0 & 0 & 1
\end{pmatrix}
\otimes
\begin{pmatrix}
1 & 0 & 0 \\
1 & 1 & 0 \\
0 & 0 & 1
\end{pmatrix}
\otimes
\begin{pmatrix}
0 & 1 \\
1 & 0
\end{pmatrix}.
\end{equation}
The $(4,3,2)$ factorization of this matrix leads to the skeleton shown in \Cref{fig: skeleton}. The final step simply consists in connecting the vertices (i.e. the black dots in \Cref{fig: skeleton}) according to the connectivity encoded in $A$. The off-diagonal entries in $A_1$, $A_2$ and $A_3$ may model long, medium and short-distance interactions between continental, regional and local communities, respectively. \Cref{fig: kronecker_graph_1} faithfully captures this interpretation. Indeed, \Cref{fig: kronecker_graph_1} reveals interactions between the first, second and third continental communities, in agreement with the off-diagonal entries of $A_1$. We also easily identity a connection between the first and second regional communities within each continental community, which results from the off-diagonal entry in $A_2$. Overall, the algorithm satisfactorily uncovers the salient features of the network. The same cannot be said of Matlab's built-in visualization algorithms, which tend to isolate connected components or pick up other important properties of the network, albeit less relevant in this context.

\begin{figure}[H]
     \centering
     \begin{subfigure}[t]{0.48\textwidth}
    \centering
    \includegraphics[width=\textwidth]{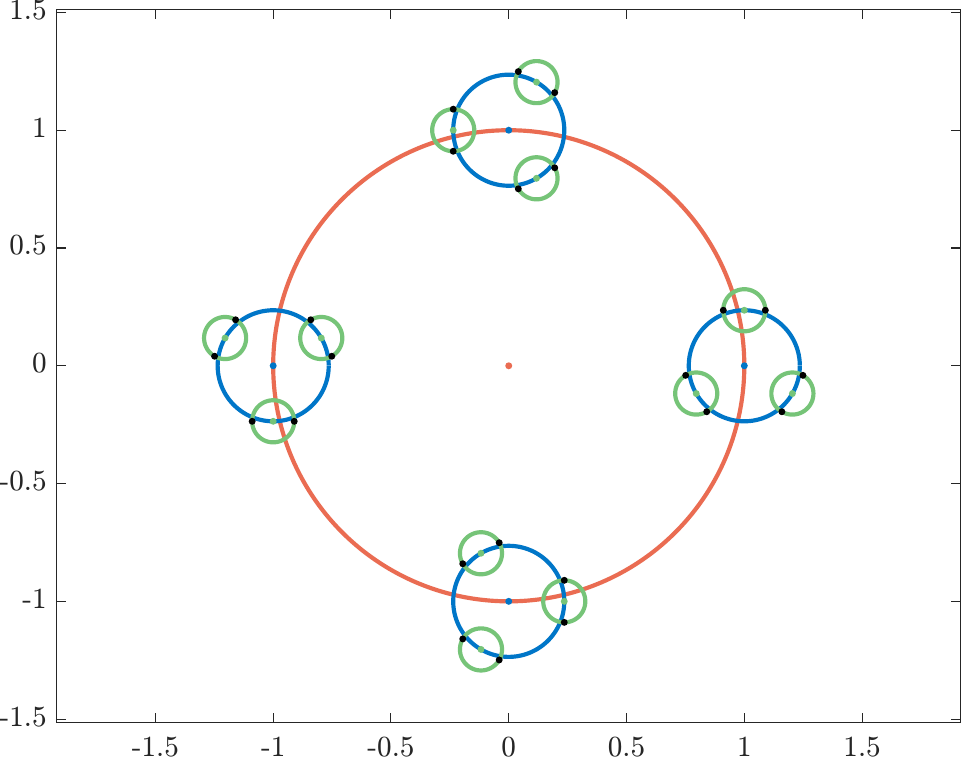}
    \caption{Skeleton: the black dots are the position of the vertices in the network}
    \label{fig: skeleton}
     \end{subfigure}
     \hfill
     \begin{subfigure}[t]{0.48\textwidth}
    \centering
    \includegraphics[width=\linewidth]{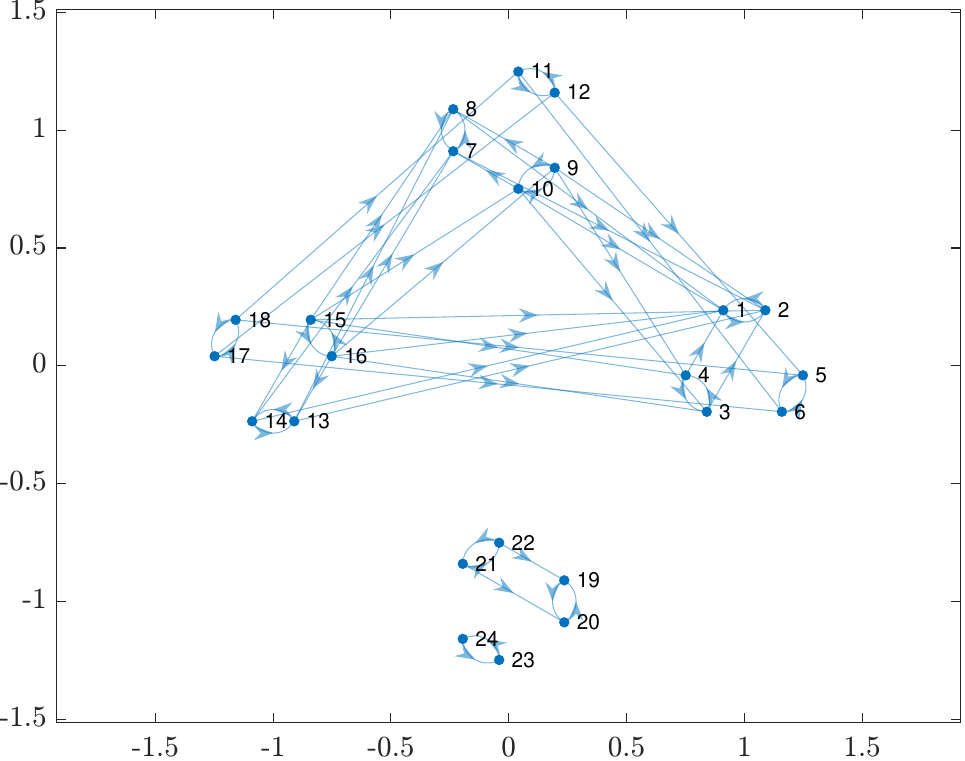}
    \caption{Graph}
    \label{fig: kronecker_graph_1}
     \end{subfigure}
     \hfill
    \caption{Kronecker graph visualization for \eqref{eq: adjacency_1}}
    \label{fig: application_2}
\end{figure}

To further highlight the capabilities of our algorithm, we test it on one of the examples presented in \cite[Figure 3]{leskovec2010kronecker}. The adjacency matrix is defined as $A=\bigotimes_{i=1}^3 A_i$, with the arrowhead matrices
\begin{equation*}
A_i=
\begin{pmatrix}
1 & 1 & 1 & 1 \\
1 & 1 & 0 & 0 \\
1 & 0 & 1 & 0 \\
1 & 0 & 0 & 1
\end{pmatrix}
\quad i=1,2,3.
\end{equation*}
\Cref{fig: kronecker_graph_2} visualizes the graph by exploiting its Kronecker structure. Once again, the layout helps identify inter-community interactions. By construction, the edges are concentrated along specific directions, which allows easily extracting useful patterns. For large networks, it might be useful to alter the color shading or transparency of the edges depending on the nature of the interactions they model (e.g. long, medium or short-distance interactions). In a computer graphics tool, one may also zoom in as much as needed to view interactions within local communities.

In fact, our visualization algorithm is applicable regardless of the graph structure, provided the sizes are given. However, it is usually ill-suited unless the graph is ``close'' to a Kronecker graph. The method described in \Cref{se: decomposable_matrices} allows identifying a Kronecker structure for a \emph{fixed} adjacency matrix but does not immediately tackle the graph factorization problem. Nevertheless, it may serve as a building block within other strategies, as for example in \cite{calderoni2023heuristic}.

\begin{figure}[H]
    \centering
    \includegraphics[width=0.5\linewidth]{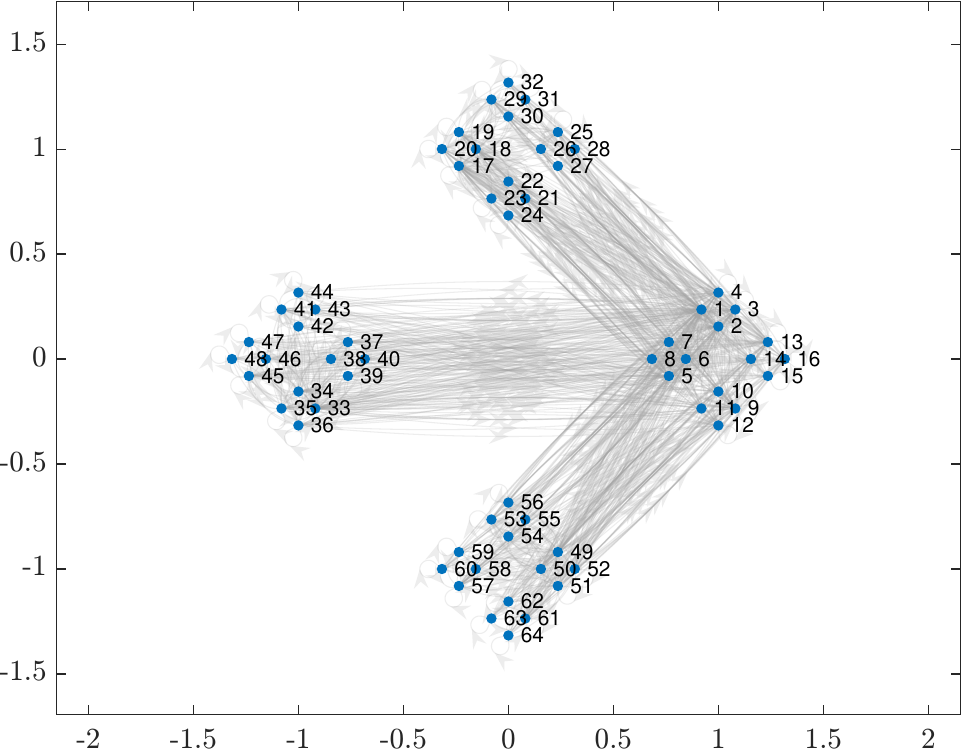}
    \caption{Kronecker graph for the first example in \cite[Figure 3]{leskovec2010kronecker}}
    \label{fig: kronecker_graph_2}
\end{figure}

\subsection{Quantum computing}
\label{se: quantum_computing}
Quantum computing is a fairly recent computational paradigm built from concepts in quantum mechanics. Contrary to classical computing, its basic unit of information, the qubit, may live in a state of superposition, a linear combination of two states. Its greatest promise lies in solving complex problems that would otherwise be infeasible on a classical computer in any reasonable time. One prominent example is the groundbreaking work of Shor \cite{shor1994algorithms,shor1999polynomial} for integer factorizations and the potential implications in cryptography. In recent years, quantum computing has gained considerable momentum, hoping for similar achievements in solving other problems. For an introduction to quantum computing, readers may consult \cite{kaye2006introduction,nielsen2010quantum}, among many other references. For a $d$-qubit system, quantum operations are defined through unitary operators (or matrices) $U \in U(2^d)$, called \emph{gates}, acting on state vectors $\ket{\psi} \in \mathbb{C}^{2^d}$. Hereafter, $U(n)$ denotes the group of unitary matrices of size $n$ and we introduce the normalized Hilbert–Schmidt (or trace) inner product $\langle A,B \rangle = \frac{1}{n}\trace(A^*B)$ with induced norm $\|A\|_H = \sqrt{\langle A, A \rangle}$. A quantum gate is called separable if its matrix representation is decomposable and is called entangling otherwise. The most obvious instance of an entangling gate is the controlled-not (or CNOT) gate
\begin{equation*}
    \begin{pmatrix}
        1 & 0 & 0 & 0 \\
        0 & 1 & 0 & 0 \\
        0 & 0 & 0 & 1 \\
        0 & 0 & 1 & 0
    \end{pmatrix}.
\end{equation*}
In the context of quantum computing, entanglement is a desirable property for reproducing quantum phenomena and many authors have introduced measures for quantifying and maximizing it \cite{kraus2001optimal,balakrishnan2010entangling}. A related question is to identify separable gates. Exact separability is rather uncommon in quantum systems since it boils down to local operations on independent subsystems. However, quantum gates are commonly defined through sequences of unitary operations such that $U=U_1U_2\dots U_k$, where $U_i \in U(2^d)$ for all $i=1,\dots,k$. Although $U$ is rarely separable, the gates $U_i$ constituting the product often are. Thus, one may try identifying separability within a subsequence of operations. For a general unitary matrix, an analytical method, related to the so-called Schmidt decomposition \cite{nielsen2010quantum}, is to first decompose $U$ in the Pauli basis. For single qubit systems, the Pauli basis is $\{\sigma_1,\sigma_2,\sigma_3,\sigma_4\}$, where
\begin{equation*}
    \sigma_1 =
    \begin{pmatrix}
        1 & 0 \\
        0 & 1
    \end{pmatrix},
    \quad
    \sigma_2 =
    \begin{pmatrix}
        0 & 1 \\
        1 & 0
    \end{pmatrix},
    \quad
    \sigma_3 =
    \begin{pmatrix}
        0 & -i \\
        i & 0
    \end{pmatrix},
    \quad
    \sigma_4 =
    \begin{pmatrix}
        1 & 0 \\
        0 & -1
    \end{pmatrix}.
\end{equation*}
For $2$-qubit systems, $\{\sigma_i \otimes \sigma_j\}_{i,j=1}^4$ forms an orthonormal basis of $U(4)$ for the normalized Hilbert–Schmidt inner product. Consequently, any $U \in U(4)$ may be decomposed as
\begin{equation*}
    U = \sum_{i,j=1}^4 \alpha_{ij} (\sigma_i \otimes \sigma_j)
\end{equation*}
with coefficients $\alpha_{ij} = \langle \sigma_i \otimes \sigma_j, U \rangle$. It follows that $U$ is separable if and only if $\alpha_{ij}=a_i b_j$. In other words, the coefficient matrix $(\alpha_{ij})_{i,j=1}^4$ is rank-$1$. Generalizing this method to $d$-qubit systems with $d \geq 2$ is straightforward. However, it is analogous to Van Loan's algorithm with the sizes $\bm{n}=(2,2,\dots,2) \in \mathbb{N}^d$ and merely tests separability in the Pauli basis instead of the canonical one. Moreover, even if the test fails, the gate may still be separable for different sizes of the factor matrices. In this section, we apply our algorithm to the matrix representation of the operator to classically determine whether it is separable. Similarly to the example in \Cref{se: space_time_iga}, our algorithm operates on the sparsity pattern for determining candidate sizes before attempting a tensor decomposition. We consider the synthetic gate $V$ encoding the circuit shown in \Cref{fig: circuit}. For an introduction to quantum circuits, interested readers may refer to \cite{kaye2006introduction,nielsen2010quantum}. The symbols in the circuit represent single-qubit or multi-qubit gates but their definition is irrelevant here.

\begin{figure}[H]
    \centering
    \includegraphics[width=0.25\linewidth]{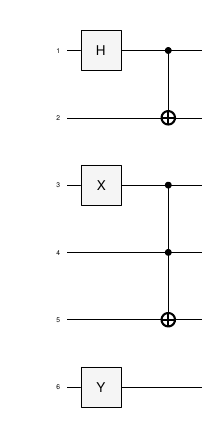}
    \caption{Synthetic quantum circuit encoded in $V$}
    \label{fig: circuit}
\end{figure}

Clearly, some of the wires in the circuit (representing qubits) do not interact and the underlying gate must have a Kronecker factorization. Indeed, our algorithm finds a unique $(4,8,2)$ prime factorization for the sparsity pattern. After supplying those sizes to a tensor decomposition algorithm, we obtain the Kronecker factorization
\begin{equation*}
    V = V_1 \otimes V_2 \otimes V_3.
\end{equation*}
Since $V$ is unitary, the factor matrices $V_i$ have orthogonal but not necessarily orthonormal columns. However, normalizing them as $V_i/\|V_i\|_H$ ensures they become unitary. Nevertheless, we must stress that this and other analytical methods can only solve relatively small problem sizes, much smaller than those quantum computing is expected to handle. Beyond analytical methods, in the quantum computing literature, Harrow and Montanaro \cite{harrow2013testing} proposed the product-state test, a probabilistic test for determining whether a gate is separable.

\section{Conclusion}
\label{se: conclusion}
In this article, we have presented a theory and algorithm for finding all possible Kronecker factorizations of a large sparse binary matrix. When encoding the sparsity pattern of real or complex matrices, factorizing binary matrices implicitly suggests suitable sizes for the factor matrices in (approximate) Kronecker factorizations. Such sparsity-informed guesses may produce exceedingly good approximations, particularly for system matrices stemming from PDE discretizations. In other applications, sparse binary matrices may represent graph adjacency matrices and factorizing them may uncover some latent structure in a network. We have subsequently proposed a graph visualization algorithm that exploits this structure to faithfully depict the nature of the interactions within the network. Although most real networks are certainly not factorizable, some are nevertheless accurately modeled as Kronecker graphs.

Regardless of their origin, binary matrices may admit multiple distinct Kronecker factorizations. For visualizing them, we have constructed a decomposition graph depicting the number of factorizations, their length and the sizes of the factor matrices entering each decomposition. Extending our framework to rectangular factor matrices is possible but quite irrelevant to the applications considered in this work, where decompositions with square factor matrices were always sought. Nevertheless, more general decompositions might have applications elsewhere and remain an interesting problem from a theoretical perspective. Even in the square case, the theoretical possibilities are way ahead of the practical reality. For instance, we are not even aware of a real network admitting multiple factorizations, although we cannot a priori exclude it.

\end{document}